\documentclass[a4paper]{amsart}
\usepackage{amsmath,amssymb}

\newtheorem{theorem}{Theorem}[section]
\newtheorem{lemma}[theorem]{Lemma}
\newtheorem{corollary}[theorem]{Corollary}
\newtheorem{proposition}[theorem]{Proposition}

\theoremstyle{definition}
\newtheorem{example}[theorem]{Example}
\newtheorem{definition}[theorem]{Definition}

\newcommand{\N}{\mathbb N}
\newcommand{\Z}{\mathbb Z}

\newcommand{\BF}{\text{\rm BF}}

\newcommand{\red}{{\text{\rm red}}}

\newcommand{\op}{\text{\rm op}}
\newcommand{\DP}{\negthinspace : \negthinspace}
\newcommand{\DPl}{\negthinspace :_l \negthinspace}
\newcommand{\DPr}{\negthinspace :_r \negthinspace}
\newcommand{\cdv}{\negthinspace \cdot_v \negthinspace}

\DeclareMathOperator{\spec}{spec}

\numberwithin{equation}{section}

\renewcommand{\t}{\, | \,}
\newcommand{\tl}{\, |_l \,}
\newcommand{\tr}{\, |_r \,}

\begin{document}

\address{Institut f\"ur Mathematik und Wissenschaftliches Rechnen \\
Karl-Franzens-Universit\"at Graz \\
Heinrichstra\ss e 36\\
8010 Graz, Austria} \email{alfred.geroldinger@uni-graz.at}

\author{Alfred Geroldinger}

\thanks{This work was supported by
the Austrian Science Fund FWF, Project Number P21576-N18}

\keywords{non-commutative Krull monoids, divisor theory, divisor
homomorphism, non-unique factorizations}

\subjclass[2010]{20M12, 20M13, 16U30, 16H10, 13F05}

\begin{abstract}
A (not necessarily commutative) Krull monoid---as introduced by
Wauters---is defined as a completely integrally closed monoid
satisfying the ascending chain condition on  divisorial two-sided
ideals. We study the structure of these Krull monoids, both with
ideal theoretic and with divisor theoretic methods. Among others we
characterize normalizing Krull monoids by divisor theories. Based on
these results we give a criterion for a  Krull monoid to be a
bounded factorization monoid, and we provide arithmetical finiteness
results in case of normalizing Krull monoids with finite Davenport
constant.
\end{abstract}

\title[Non-commutative  Krull monoids:  a divisor theoretic approach]{Non-commutative  Krull monoids: \\ a divisor theoretic approach and their arithmetic}

\maketitle

\bigskip
\section{Introduction}
\bigskip

The arithmetic concept of a divisor theory has its origin in early
algebraic number theory. Axiomatic approaches to more general
commutative domains and monoids were formulated by Clifford
\cite{Cl38}, by Borewicz and \v{S}afarevi\v{c} \cite{Bo-Sa66}, and
then by Skula \cite{Sk70} and Gundlach \cite{Gu72}. The theory of
divisorial ideals was developed in the first half of the 20th
century by Pr\"ufer, Krull and Lorenzen \cite{Pr32, Kr36r, Kr36f,
Kr37, Lo39}, and its presentation in the book of  Gilmer
\cite{Gi72a} strongly influenced the development of multiplicative
ideal theory. The concept of a commutative Krull monoid (defined as
completely integrally closed commutative monoids satisfying the
ascending chain condition on divisorial ideals) was introduced by
Chouinard \cite{Ch81}  1981 in order to study the Krull ring
property of commutative semigroup rings.

Fresh impetus   came from the theory of non-unique factorizations in
the 1990s.  Halter-Koch observed that  the concept of monoids
with divisor theory coincides with the concept of Krull monoids
\cite{HK90b}, and Krause \cite{Kr89} proved that a commutative
domain is a Krull domain if and only if its multiplicative monoid of
non-zero elements is a Krull monoid. Both, the concepts of divisor
theories and of Krull monoids, were widely generalized, and a
presentation can be found in the monographs \cite{HK98, Ge-HK06a}
(for a recent survey see \cite{HK10b}).

The search for classes of non-commutative rings having an
arithmetical ideal theory---generalizing the classical theory of
commutative Dedekind and Krull domains---was started with the
pioneering work of  Asano \cite{As39a, As49a, As50a, As-Mu53a}.
It lead to a theory of Dedekind-like rings, including  Asano prime
rings and Dedekind prime rings. Their ideal theory and also their
connection with classical maximal orders over Dedekind domains in
central simple algebras is presented in \cite{Mc-Ro01a}.

From the 1970s on a large number of concepts of non-commutative Krull rings has been
introduced (see the contributions of Brungs, Bruyn, Chamarie,
Dubrovin, Jespers, Marubayashi, Miyashita, Rehm and Wauters, cited
in the references). Always the commutative situation was used as a
model, and all these generalizations include Dedekind prime rings as
a special case (see the survey of  Jespers \cite{Je86a}, and Section
\ref{5} for more details). The case of semigroup rings has received
special attention, and the reader may want to consult the monograph
of Jespers and Okni{\'n}ski \cite{Je-Ok07a}.

In 1984  Wauters \cite{Wa84a} introduced non-commutative Krull
monoids generalizing the concept of Chouinard to the non-commutative
setting. His focus was on normalizing Krull monoids, and he showed,
among others, that a prime polynomial identity ring is a
Chamarie-Krull ring if and only if its monoid of regular elements is
a Krull monoid (see  Section \ref{5}).

In the present paper we study non-commutative Krull monoids in the
sense of Wauters, which are defined as completely integrally closed
monoids satisfying the ascending chain condition on
divisorial two-sided ideals. In Section \ref{3} we develop the theory of
divisorial two-sided ideals in analogy to the commutative setting (as it is
done in \cite{HK98, Ge-HK06a}). In Section \ref{4} we introduce
divisor theoretic concepts, and provide a characterization of
normalizing Krull monoids in divisor theoretic terms (Theorem
\ref{4.13}). Although many results and their proofs are very similar
either to those for commutative  monoids or to those for
non-commutative rings, we provide full proofs. In Section \ref{5} we
discuss examples of commutative and non-commutative Krull monoids
with an emphasis on the connection to ring theory. The existence of
a suitable divisor homomorphism is crucial for the investigation of
arithmetical finiteness properties in commutative Krull monoids (see
\cite[Section 3.4]{Ge-HK06a}). Based on the results in Sections {3}
and \ref{4} we can do  some first steps towards a better
understanding of the arithmetic of non-commutative Krull monoids.
Among others, we generalize the concept of  transfer homomorphisms,
give a criterion for a  Krull monoid to be a  bounded-factorization
monoid, and we provide arithmetical finiteness results in case of
normalizing Krull monoids with finite Davenport constant (Theorem
\ref{6.5}).

\bigskip
\section{Basic concepts} \label{2}
\bigskip

Let $\mathbb N$ denote the set of positive integers, and let
$\mathbb N_0 = \mathbb N \cup \{0\}$. For integers $a, b \in \Z$, we
set $[a, b ] = \{ x \in \Z \mid a \le x \le b]$. If $A, B$ are sets, then $A \subset B$ means that $A$ is contained in $B$ but may be equal to $B$.

By a {\it semigroup} we always mean an associative semigroup with
unit element. If not denoted otherwise, we  use multiplicative
notation. Let $H$ be a semigroup. We say that $H$ is {\it
cancellative} if for all elements $a, b, c \in H$, the equation $ab =
ac$ implies $b=c$ and the equation $ba = ca$ implies $b = c$.
Clearly, subsemigroups of groups are cancellative. A group $Q$ is
called a {\it left quotient group} of $H$ (a {\it right quotient
group} of $H$, resp.) if $H \subset Q$ and every element of $Q$ can
be written in the form $a^{-1}b$ with $a, b \in H$ (or in the form
$b a^{-1}$, resp.).

We say that $H$ satisfies the {\it right Ore condition} ({\it left
Ore condition}, resp.) if $aH \cap bH \ne \emptyset$ ($Ha \cap Hb
\ne \emptyset$, resp.) for all $a, b \in H$. A cancellative semigroup
has a left quotient group if and only if it satisfies the left Ore
condition, and if this holds, then the left quotient group is unique
up to isomorphism (see \cite[Theorems 1.24 and 1.25]{Cl-Pr64}).
Moreover, a semigroup is embeddable in a group if and only if it is
embeddable in a left (resp. right) quotient group (see
\cite[Section 12.4]{Cl-Pr67}).

If $H$ is cancellative and satisfies the left and right Ore
condition, then every right quotient group $Q$ of $H$ is also a left
quotient group and conversely. In this case, $Q$ will simply be
called a {\it quotient group} of $H$ (indeed, if $Q$ is a right
quotient group and $s = a x^{-1} \in Q$ with $a, x \in H$, then the
left Ore condition implies the existence of $b, y \in H$ such that
$y a = b x$ and hence $s = a x^{-1} = y^{-1}b$; thus $Q$ is a left
quotient group).

\smallskip
Throughout this paper, a {\it monoid} means a cancellative semigroup
which satisfies the left and the right Ore condition, and every
monoid homomorphism $\varphi \colon H \to D$ satisfies $\varphi
(1_H) = 1_D$.

Let $H$ be a  monoid. We denote by $\mathsf q (H)$ a quotient group
of $H$. If $\varphi \colon H \to D$ is a monoid homomorphism, then
there is a unique homomorphism $\mathsf q (\varphi) \colon \mathsf q
(H) \to \mathsf q (D)$ satisfying $\mathsf q (\varphi) \mid H =
\varphi$. If $S$ is a semigroup with $H \subset S \subset \mathsf q
(H)$, then $S$ is cancellative, $\mathsf q (H)$ is a quotient group
of $S$, and hence $S$ is a monoid. Every such monoid $S$ with $H
\subset S \subset \mathsf q (H)$ will be called an {\it overmonoid}
of $H$. Let $H^{\op}$ denote the {\it opposite monoid} of $H$
($H^{\op}$ is a semigroup  on the set $H$, where multiplication
$H^{\op} \times H^{\op} \to H^{\op}$ is defined by $(a,b) \mapsto
ba$ for all $a, b \in H$; clearly, $H^{\op}$ is a monoid in the
above sense). We will encounter many statements on left and right
ideals (quotients, and so on) in the monoid $H$. Since every
right-statement {\bf (r)} in $H$ is a left-statement {\bf (l)} in
$H^{\op}$, it will always be sufficient to prove the left-statement.

Let $a, b \in H$. The element $a$ is said to be {\it invertible} if
there exists an $a' \in H$ such that $a a' = a' a = 1$. The set of
invertible elements of $H$ will be denoted by $H^{\times}$, and it
is a subgroup of $H$. We say that $H$ is reduced if $H^{\times} =
\{1\}$. A straightforward calculation shows that $aH = bH$ if and
only if $a H^{\times} = b H^{\times}$.

We say that $a$ is a {\it left divisor} ({\it right divisor}, resp.)
if $b \in aH$ ($b \in Ha$, resp.), and we denote this by $a \tl b$
($a \tr b$, resp.). If $b \in aH \cap Ha$, then we say that $a$ is a
{\it divisor} of $b$, and then we write $a \t b$.

The element $a$ is called an  {\it atom}  if $a
      \notin H^{\times}$ and, for all $u, v  \in H$,  $a = u v$ implies
      $u \in H^{\times}$ or $v \in H^{\times}$. The set of atoms of $H$
      is denoted by $\mathcal A (H)$. $H$ is said to be \ {\it atomic}
      if every $u \in H \setminus H^{\times}$ is a product of finitely many atoms of $H$

For a set $P$, we denote by $\mathcal F (P)$ the {\it free abelian
monoid} with basis $P$. Then every $a \in \mathcal F (P)$ has a
unique representation in the form
\[
a = \prod_{p \in P} p^{\mathsf v_p (a)} \,, \quad \text{where} \quad
\mathsf v_p (a) \in \mathbb N_0 \quad \text{and} \quad \mathsf v_p
(a) = 0 \quad \text{for almost all} \quad p \in P \,,
\]
and we call $|a| = \sum_{p \in P} \mathsf v_p (a) \in \mathbb N_0$
the {\it length} of $a$. If $H = \mathcal F (P)$ is free abelian
with basis $P$, then $H$ is reduced, atomic with $\mathcal A (H) =
P$ and $\mathsf q (H) \cong (\Z^{(P)}, +)$. We use all notations and conventions concerning greatest common divisors in commutative monoids as in \cite[Chapter 10]{HK98}.

\bigskip
\section{Divisorial ideals in monoids} \label{3}
\bigskip

In this section we develop the theory of divisorial ideals in
monoids as far as it is needed for the divisor theoretic approach in
Section \ref{4} and the arithmetical results in Section \ref{6}. An ideal will always be a two-sided ideal. We
follow the presentation in the commutative setting (as given in
\cite{HK98, Ge-HK06a}) with the necessary adjustments. The
definition of a Krull monoid (as given in Definition \ref{3.11}) is
due to  Wauters \cite{Wa84a}. For Asano orders $H$ (see Section
\ref{5}), the commutativity of the group $\mathcal F_v (H)^{\times}$
(Proposition \ref{3.12}) dates back to the classical papers of Asano
and can also be found in \cite[Chapter II, $\S$ 2]{Ma-Ra80a}.

Our first step is to introduce modules (following the terminology of \cite{HK10b}),
fractional ideals and divisorial fractional ideals. Each definition
will be followed by a simple technical lemma.

\medskip
\begin{definition} \label{3.1}
Let $H$ be a monoid and $A, B \subset \mathsf q (H)$  subsets.
\begin{enumerate}
\item We say that $A$ is a {\it left module} (resp. {\it right module}) if $HA =
      A$ (resp. $AH = A$), and denote by $\mathcal M_l (H)$ (resp.
      $\mathcal M_r (H)$) the set of all left (resp. right) modules. The
      elements of $\mathcal M (H) = \mathcal M_l (H) \cap \mathcal M_r
      (H)$ are called {\it modules} (of $H$).

\smallskip
\item We set $AB = \{ a b \mid a \in A, b \in B \}$, and define the
      {\it left} and {\it right quotient} of $A$ and $B$ by
      \[
      (A \DPl B ) = \{ x \in \mathsf q (H) \mid x B \subset A \} \quad
      \text{and} \quad (A \DPr B ) = \{ x \in \mathsf q (H) \mid  B x
      \subset A \} \,.
      \]
      If $B = \{b\}$, then $(A \DPl b) = (A \DPl B)$ and $(A \DPr b)
      = (A \DPr B)$.
\end{enumerate}
\end{definition}

\smallskip
The following lemma gathers some simple properties which will be
used without further mention (most of them have a symmetric left or
right variant).

\medskip
\begin{lemma} \label{3.2}
Let $H$ be a monoid, $A, B, C \subset \mathsf q (H)$  subsets, and
$c \in H$.
\begin{enumerate}
\item $(A \DPl c) = Ac^{-1}$, $(cA \DPl B) = c(A \DPl B)$, $(Ac \DPl
      B ) = (A \DPl Bc^{-1})$, and $(A \DPl cB) = c^{-1}(A \DPl B)$.

\smallskip
\item $( A \DPl B) = \bigcap_{b \in B} (A \DPl b) = \bigcap_{b \in B} Ab^{-1}$.

\smallskip
\item $(A \DPl BC) = \big( (A \DPl C) \DPl B\big)$ and $\big( (A
      \DPl B) \DPr C \big) = \big( (A \DPr C) \DPl B \big)$.

\smallskip
\item $
      A \subset \big( H \DPl (H \DPr A) \big) \ = \ \bigcap_{c \in \mathsf q (H), A \subset Hc} Hc  \quad  \quad \text{and}
      \quad A \subset \big( H \DPr (H \DPl A)\big) \ = \ \bigcap_{c \in \mathsf q (H), A \subset cH} c H \,.
      $

\smallskip
\item \begin{enumerate}
      \item If $A \in \mathcal M_l (H)$, then $(A \DPl B) \in \mathcal M_l
            (H)$.

      \item If $A \in \mathcal M_r (H)$, then $(A \DPl B) = (A \DPl BH)$.

      \item If $B \in \mathcal M_l (H)$, then $(A \DPl B) \in
            \mathcal M_r (H)$.
      \end{enumerate}
\end{enumerate}
\end{lemma}

\begin{proof}
We verify only the statements 3. and 4., as the remaining ones
follow immediately from the definitions.

\smallskip
3.  We have
\[
(A \DPl B C) = \{ x \in \mathsf q (H) \mid x B C \subset A \} = \{ x
\in \mathsf q (H) \mid x B \subset (A \DPl C) \} = \big( (A \DPl C)
\DPl B \big) \,,
\]
and
\[
\begin{aligned}
\big( (A \DPl B) \DPr C \big) & = \{ x \in \mathsf q (H) \mid Cx
\subset (A \DPl B) \} = \{ x \in \mathsf q (H) \mid CxB \subset A \}
\\
& = \{ x \in \mathsf q (H) \mid xB \subset (A \DPr C) \} = \big( (A
\DPr C) \DPl B \big) \,.
\end{aligned}
\]

\smallskip
4. We check only the first equality. Let $a$ be an element of the given
intersection. We have to show that $a (H \DPr A) \subset H$, whence
for all $b \in (H \DPr A)$ we have to verify that $ab \in H$. If $b
\in (H \DPr A)$, then $A b \subset H$ implies that $A \subset H
b^{-1}$. Thus we obtain that
\[
a \in \bigcap_{c \in \mathsf q (H), A \subset Hc} Hc \ \subset \
Hb^{-1} \,,
\]
and thus $ab \in H$. Conversely, suppose that $a \in \big( H \DPl (H
\DPr A) \big)$. We have to verify that $a \in Hc$ for all $c \in
\mathsf q (H)$ with $A \subset Hc$. If $A \subset Hc$, then $A
c^{-1} \subset H$ implies that $c^{-1} \in (H \DPr A)$. Thus we get
$ac^{-1} \in H$ and $a \in Hc$.
\end{proof}

\medskip
\begin{definition} \label{3.3}
Let $H$ be a monoid and $A \subset \mathsf q (H)$ a subset.
Then $A$ is said to be
\begin{itemize}
\item {\it left $($resp. right$)$ $H$-fractional} \ if there exist $a \in H$ such that $Aa \subset H$ (resp. $aA \subset H )$.

\item $H$-{\it fractional} \ if $A$ is left and right $H$-fractional.

\item a  {\it fractional left $($resp. right$)$ ideal} (of $H$) \ if $A$ is left $H$-fractional and a left module (resp.  right $H$-fractional and  a right module $)$.

\item a {\it left $($resp. right$)$ ideal} (of $H$) \ if $A$ is a fractional left ideal (resp. right ideal) and $A \subset H$.

\item a {\it $($fractional$)$ ideal} \ if $A$ is a (fractional) left and
right ideal.
\end{itemize}
We denote by $\mathcal
F_s (H)$ the set of fractional ideals of $H$, and by $\mathcal I_s
(H)$ the set of ideals of $H$.
\end{definition}

\smallskip
Note that the empty set is an ideal of $H$. Let $A \subset \mathsf q (H)$ be a subset. Then $A$ is
\begin{itemize}
\item left $H$-fractional if and only if $(H \DPr A) \ne
      \emptyset$ if and only if $(H \DPr A) \cap H \ne
      \emptyset$.

\item right $H$-fractional if and only if $(H \DPl A) \ne \emptyset$ if and only if $(H \DPl A) \cap H \ne \emptyset$.
\end{itemize}
Thus, if $A$ is non-empty, then Lemma \ref{3.2} (Items 4. and 5.) shows that $(H \DPl A)$ is a fractional left ideal and $(H \DPr A)$ is a fractional right ideal.

\medskip
\begin{lemma} \label{3.4}
Let $H$ be a monoid.
\begin{enumerate}
\item If $(\mathfrak a_i)_{i \in I}$ is a family of fractional left ideals $($resp. right ideals or ideals$)$ and $J \subset I$ is finite,
      then $\bigcap_{i \in I} \mathfrak a_i$ and $\prod_{i \in J} \mathfrak a_i$
      are fractional left ideals $($resp. right ideals or ideals$)$.

\smallskip
\item Equipped with usual multiplication, $\mathcal F_s (H)$ is a semigroup with unit element $H$.

\smallskip
\item If $\mathfrak a \in \mathcal F_s (H)^{\times}$, then $(H \DP_l \mathfrak a) \mathfrak a = H = \mathfrak a (H \DPr \mathfrak a)$ and
      $(H \DPl \mathfrak a) = (H \DPr \mathfrak a) \in \mathcal F_s (H)$.

\smallskip
\item For every $a \in \mathsf q (H)$, we have $(H \DPl aH) = Ha^{-1}$, $(H \DPr Ha) =
      a^{-1}H$, $\big( H \DPl (H \DPr Ha) \big) = Ha$ and $\big( H \DPr (H
      \DPl aH) \big) = aH$.

\smallskip
\item If $A \subset \mathsf q (H)$, then
      $\big( H \DPl (H \DPr A) \big)$ is a fractional
      left ideal and $\big( H \DPr (H \DPl A)\big)$ is a fractional
      right ideal.

\smallskip
\item If $A \subset \mathsf q (H)$, $\mathfrak a = (H \DPl A)$ and $\mathfrak b = (H \DPr A)$, then $\mathfrak a =  \big( H \DPl (H \DPr
      \mathfrak a)\big)$ and $\mathfrak b = \big( H \DPr (H \DPl \mathfrak b)\big)$.
\end{enumerate}
\end{lemma}

\begin{proof}
1. Since $\bigcap_{i \in I} \mathfrak a_i \subset \mathfrak a_j$,  $\prod_{i \in J} \mathfrak a_i \subset \mathfrak a_j$ for some $j \in J$ and subsets of left (resp. right)
$H$-fractional sets are left (resp. right) $H$-fractional,
the given intersection and product are left (resp. right) $H$-fractional, and then clearly they are
 fractional left
ideals (resp. fractional right ideals or ideals).

\smallskip
2. Obvious.

\smallskip
3. Let $\mathfrak a \in \mathcal F_s (H)^{\times}$ and $\mathfrak b
\in \mathcal F_s (H)$ with $\mathfrak b \mathfrak a = \mathfrak a
\mathfrak b = H$. Then $\mathfrak b \subset (H \DPl \mathfrak a)$
and hence $H = \mathfrak b \mathfrak a \subset (H \DPl \mathfrak a)
\mathfrak a \subset H$, which implies that $(H \DPl \mathfrak a)
\mathfrak a = H$. Similarly, we obtain that $\mathfrak a (H \DPr
\mathfrak a) = H$, and therefore $(H \DPl \mathfrak a) = \mathfrak b
= (H \DPr \mathfrak a) \in \mathcal F_s (H)$.

\smallskip
4. Let $a \in \mathsf q (H)$. The first two equalities follow
directly from the definitions. Using them we infer that
\[
\big( H \DPl (H \DPr Ha) \big) = ( H \DPl a^{-1}H) = Ha \quad
\text{and} \quad \big( H \DPr (H  \DPl aH) \big) = (H \DPr Ha^{-1})
= aH \,.
\]

\smallskip
5. This follows from 1. and from Lemma \ref{3.2}.4.

\smallskip
6. By Lemma \ref{3.2}.4, we have $\mathfrak a \subset  \big( H \DPl (H \DPr \mathfrak a)\big)$.
Conversely, if $q \in \big( H \DPl (H \DPr \mathfrak a)\big)$, then
\[
q A \subset q \big( H \DPr (H \DPl A) \big)
\subset q (H \DPr \mathfrak a) \subset H \,,
\]
and hence $q \in (H \DPl A) = \mathfrak a$.
\end{proof}

\medskip
\begin{definition} \label{3.5}
Let $H$ be a monoid and $A \subset \mathsf q (H)$ a subset.
\begin{enumerate}
\item $A$ is called a {\it divisorial fractional left ideal} if \ $A = \big( H \DPl (H \DPr A) \big)$, and a {\it divisorial fractional right ideal} if \ $A = \big( H \DPr (H \DPl A) \big)$.

\smallskip
\item If  $(H \DPl A) = (H \DPr A)$, then we set $A^{-1} = (H \DP A) = (H \DPl A)$.

\smallskip
\item If $\big( H \DPl (H \DPr A)\big) = \big( H \DPr (H \DPl A)\big)$, then
we set $A_v = \big( H \DPl (H \DPr A)\big)$, and  $A$ is said to be a
{\it divisorial fractional ideal} (or a {\it fractional $v$-ideal})
if $A = A_v$. The  set of such ideals will be denoted by $\mathcal
F_v (H)$, and $\mathcal I_v (H) = \mathcal F_v (H) \cap \mathcal I_s
(H)$ is the set of {\it divisorial ideals} of $H$ (or the set of
$v$-ideals of $H$).

\smallskip
\item
Suppose that $(H \DPl \mathfrak c) = (H \DPr \mathfrak c)$ for all fractional ideals $\mathfrak c$ of $H$.
      \begin{enumerate}

      \smallskip
      \item For fractional ideals $\mathfrak a, \,\mathfrak b$ we define \ $\mathfrak a \cdv \mathfrak b = (\mathfrak a \mathfrak b)_v$, and we call $\mathfrak a \cdv \mathfrak b$ the \ $v$-{\it product} \ of $\mathfrak a$ and $\mathfrak b$.

      \smallskip
      \item A fractional $v$-ideal $\mathfrak a$ is called \ {\it $v$-invertible} \ if $\mathfrak a \cdv \mathfrak a^{-1} = \mathfrak a^{-1} \cdv \mathfrak a = H$. We denote by \ $\mathcal I_v^*(H)$ \ the set of all $v$-invertible $v$-ideals.

      \end{enumerate}
\end{enumerate}
\end{definition}

\smallskip
Lemma \ref{3.4}.5 shows that a divisorial fractional left ideal is
indeed a fractional left ideal, and the analogous statement holds for
divisorial fractional right ideals and for divisorial fractional
ideals. Furthermore, Lemma \ref{3.4}.4 shows that, for every $a \in
\mathsf q (H)$,   $Ha$ is a divisorial fractional left ideal. We
will see that the assumption of Definition \ref{3.5}.4 holds in
completely integrally closed monoids (Definition \ref{3.11}) and in
normalizing monoids (Lemma \ref{4.5}).

\medskip
\begin{lemma} \label{3.6}
Suppose that $(H \DPl \mathfrak c) = (H \DPr \mathfrak c)$ for all fractional ideals $\mathfrak c$
of $H$,  and let $\mathfrak a, \mathfrak b$ be fractional ideals of $H$.

\begin{enumerate}
\smallskip
\item We have $\mathfrak a \subset \mathfrak a_v = (\mathfrak a_v)_v$ and $(\mathfrak a_v)^{-1} = \mathfrak a^{-1} = (\mathfrak a^{-1})_v$.
      In particular,  $\mathfrak a^{-1}, \mathfrak a_v \in \mathcal F_v (H)$.

\smallskip
\item $( \mathfrak a \mathfrak a^{-1})_v =
(\mathfrak a_v \DP \mathfrak a)^{-1}$.

\smallskip
\item If $\mathfrak a, \mathfrak b \in \mathcal F_v (H)$, then $\mathfrak a \cdv \mathfrak b \in \mathcal F_v (H)$ and $\mathfrak a \cap \mathfrak b \in \mathcal F_v (H)$, and if $\mathfrak a, \mathfrak b \in \mathcal I_v (H)$, then $\mathfrak a \cdv \mathfrak b \in \mathcal I_v (H)$, $\mathfrak a \cap \mathfrak b \in \mathcal I_v (H)$, and $\mathfrak a \cdv \mathfrak b \subset \mathfrak a \cap \mathfrak b$.

\smallskip
\item If $d \in \mathsf q (H)$ with $d \mathfrak a \subset \mathfrak
      b$, then $d \mathfrak a_v \subset \mathfrak b_v$. Similarly,
      $\mathfrak a d \subset \mathfrak b$ implies that $\mathfrak a_v d \subset \mathfrak
      b$.

\smallskip
\item We have $(\mathfrak a \mathfrak b)_v = (\mathfrak a_v
      \mathfrak b)_v = (\mathfrak a_v \mathfrak b_v)_v$.

\smallskip
\item Equipped with $v$-multiplication, $\mathcal F_v (H)$ is a semigroup with unit element $H$, and $\mathcal I_v (H)$ is a subsemigroup.
      Furthermore, if $\mathfrak a \in \mathcal F_v (H)$, then $\mathfrak a$ is $v$-invertible if and only if $\mathfrak a \in \mathcal F_v
      (H)^{\times}$, and hence $\mathcal I_v^* (H) = \mathcal I_v
      (H) \cap \mathcal F_v (H)^{\times}$.
\end{enumerate}
\end{lemma}

\begin{proof}
1. By Lemma \ref{3.4}.5, we have $\mathfrak a \subset   \mathfrak
a_v$. Therefore it follows that
\[
\big[ ( \mathfrak a^{-1})^{-1} \big]^{-1} = (\mathfrak a_v)^{-1}
\subset \mathfrak a^{-1} \subset (\mathfrak a^{-1})_v = \big[ (
\mathfrak a^{-1})^{-1} \big]^{-1} \,,
\]
hence $(\mathfrak a_v)^{-1} = \mathfrak a^{-1} = (\mathfrak
a^{-1})_v$ and $(\mathfrak a_v)_v = \big( (\mathfrak a_v)^{-1}
\big)^{-1} = (\mathfrak a^{-1})^{-1} = \mathfrak a_v$.

\smallskip
2.
Using Lemma \ref{3.2}.3 we infer that
\[
(\mathfrak a \mathfrak a^{-1})^{-1} = (H \DP \mathfrak a \mathfrak
a^{-1}) = \big( ( H \DP \mathfrak a^{-1}) \DP \mathfrak a \big) = (
\mathfrak a_v \DP \mathfrak a) \,,
\]
and hence $(\mathfrak a \mathfrak a^{-1})_v = ( \mathfrak a_v \DP
\mathfrak a)^{-1}$.

\smallskip
3. Let $\mathfrak a, \mathfrak b \in \mathcal F_v (H)$. Then
$\mathfrak a \cdv \mathfrak b = ( \mathfrak a \mathfrak b)_v$ is a
divisorial fractional ideal by 1. Clearly, we have $\mathfrak a \cap
\mathfrak b \subset (\mathfrak a \cap \mathfrak b)_v \subset
\mathfrak a_v \cap \mathfrak b_v = \mathfrak a \cap \mathfrak b$.
The remaining statements are clear.

\smallskip
4. If $d \mathfrak a \subset \mathfrak b$, then we get
\[
\begin{aligned}
d \mathfrak a_v & = \ d  \bigcap_{c \in \mathsf q (H), \mathfrak a
\subset cH} c H \ = \ \bigcap_{c \in \mathsf q (H), d \mathfrak a
\subset dcH} d c H  = \bigcap_{e \in \mathsf q (H), d \mathfrak a
\subset e H} e  H \\
 & = \big( H \DPr (H \DPl d \mathfrak a) \big)
\subset \big( H \DPr (H \DPl \mathfrak b) \big) = \mathfrak b_v \,.
\end{aligned}
\]
If $\mathfrak a d \subset \mathfrak b$, we argue similarly.

\smallskip
5. We have $(\mathfrak a \mathfrak b)_v \subset (\mathfrak a_v
\mathfrak b)_v \subset (\mathfrak a_v \mathfrak b_v)_v$. To obtain
the reverse inclusion it is sufficient to verify that
\[
(\mathfrak a \mathfrak b)^{-1} \subset (\mathfrak a_v \mathfrak
b_v)^{-1} \,.
\]
Let $d \in (\mathfrak a \mathfrak b)^{-1}$. Then $d \mathfrak a
\mathfrak b \subset H$ and hence $d a \mathfrak b \subset H$ for all
$a \in \mathfrak a$. Then 4. implies that $d a \mathfrak b_v \subset
H_v = H$ for all $a \in \mathfrak a$ and hence $d \mathfrak a
\mathfrak b_v \subset H$. Since $\mathfrak a \mathfrak b_v$ is a
fractional ideal, it follows that $\mathfrak a \mathfrak b_v d
\subset H$ and hence $\mathfrak a b d \subset H$ for all $b \in
\mathfrak b_v$. Again 4. implies that $\mathfrak a_v b d \subset H$
for all $b \in \mathfrak b_v$ and hence $\mathfrak a_v \mathfrak b_v
d \subset H$.

\smallskip
6. Using 5. we obtain to first assertion. We provide the details for the
furthermore statement. Let $\mathfrak a \in \mathcal F_v (H)$. Then
$\mathfrak a^{-1} \in \mathcal F_v (H)$, and thus, if $\mathfrak a$
is $v$-invertible, then $\mathfrak a \in \mathcal F_v (H)^{\times}$.
Conversely, suppose that $\mathfrak a \in \mathcal F_v (H)^{\times}$
and let $\mathfrak b \in \mathcal F_v (H)$ such that $\mathfrak a
\cdv \mathfrak b = \mathfrak b \cdv \mathfrak a = H$. Then
$\mathfrak a \mathfrak b \subset H$, hence $\mathfrak b \subset (H
\DP \mathfrak a)$ and $\mathfrak a \mathfrak b \subset \mathfrak a
(H \DP \mathfrak a) \subset H$. This implies that $H = (\mathfrak a
\mathfrak b)_v \subset \mathfrak a \cdv \mathfrak a^{-1} \subset H$.
Similarly, we get $\mathfrak a^{-1} \cdv \mathfrak a = H$, and hence
$\mathfrak a$ is $v$-invertible.
\end{proof}

\smallskip
The next topic are prime ideals and their properties.

\medskip
\begin{lemma} \label{3.7}
Let $H$ be a monoid and $\mathfrak p \subset H$ an ideal. Then the
following statements are equivalent{\rm \,:}
\begin{enumerate}
\item[(a)] If $\mathfrak a, \mathfrak b \subset H$ are ideals with
           $\mathfrak a \mathfrak b \subset \mathfrak p$, then $\mathfrak a
           \subset \mathfrak p$ or $\mathfrak b \subset \mathfrak p$.

\smallskip
\item[(b)] If $\mathfrak a, \mathfrak b \subset H$ are right ideals with
           $\mathfrak a \mathfrak b \subset \mathfrak p$, then $\mathfrak a
           \subset \mathfrak p$ or $\mathfrak b \subset \mathfrak p$.

\smallskip
\item[(c)] If $\mathfrak a, \mathfrak b \subset H$ are left ideals with
           $\mathfrak a \mathfrak b \subset \mathfrak p$, then $\mathfrak a
           \subset \mathfrak p$ or $\mathfrak b \subset \mathfrak p$.

\smallskip
\item[(d)] If $a, b \in H$ with $aHb \subset \mathfrak p$, then $a
           \in \mathfrak p$ or $b \in \mathfrak p$.
\end{enumerate}
\end{lemma}

\begin{proof}
(a) \,$\Rightarrow$\, (b) \ If $\mathfrak a, \mathfrak b \subset H$
are right ideals with
           $\mathfrak a \mathfrak b \subset \mathfrak p$, then $H
           \mathfrak a, H \mathfrak b \subset H$ are ideals with $(H
           \mathfrak a)(H \mathfrak b) = H \mathfrak a \mathfrak b
           \subset H \mathfrak p = \mathfrak p$, and hence
           $\mathfrak a \subset H \mathfrak a \subset \mathfrak p$
           or $\mathfrak b \subset H \mathfrak b \subset \mathfrak
           p$.

\smallskip
(b) \,$\Rightarrow$\, (d) \ If $a, b \in H$ with $aHb \subset
\mathfrak p$, then $(aH)(bH)  \subset \mathfrak p H = \mathfrak p$,
and hence $a \in aH \subset \mathfrak p$ or $b \in bH \subset
\mathfrak p$.

\smallskip
(d) \,$\Rightarrow$\, (a) \ If $\mathfrak a \not\subset \mathfrak p$
and $\mathfrak b \not\subset \mathfrak p$, then there exist $a \in
\mathfrak a \setminus \mathfrak p$, $b \in \mathfrak b \setminus
\mathfrak p$, and hence $aHb \not\subset \mathfrak p$, which implies
that $\mathfrak a \mathfrak b \not\subset \mathfrak p$.

\smallskip
The proof of the implications (a) \,$\Rightarrow$\, (c)
\,$\Rightarrow$\, (d) \,$\Rightarrow$\, (a) runs along the same
lines.
\end{proof}

\smallskip
An ideal $\mathfrak p \subset H$ is called {\it prime} if $\mathfrak p \ne  H$ and if it
satisfies the equivalent statements in Lemma \ref{3.7}. We denote by
$s$-$\spec (H)$ the set of prime ideals of $H$, and by $v$-$\spec
(H) = s$-$\spec (H) \cap \mathcal I_v (H)$  the set of divisorial
prime ideals of $H$. Following ring theory (\cite[Definition
10.3]{La01a}), we call a subset $S \subset H$ an $m$-{\it system}
if, for any $a, b \in S$, there exists an $h \in H$ such that $ahb
\in S$. Thus Lemma \ref{3.7}.(d) shows that an ideal $\mathfrak p
\subset H$ is prime if and only if $H \setminus \mathfrak p$ is an
$m$-system.

A subset $\mathfrak m \subset H$ is called a \ $v$-{\it maximal}
$v$-{\it ideal} \ if $\mathfrak m$ is a maximal element of $\mathcal
I_v(H) \setminus \{H\}$ (with respect to the inclusion).
            We denote by $v$-$\max(H)$ \ the set of all $v$-maximal $v$-ideals of $H$.

\medskip
\begin{lemma} \label{3.8}
Suppose that $(H \DPl \mathfrak c) = (H \DPr \mathfrak c)$ for all
fractional ideals $\mathfrak c$ of $H$.
\begin{enumerate}
\item If \ $S \subset H$ is an $m$-system and $\mathfrak p$ is maximal
      in the set $\{ \mathfrak a \in \mathcal I_v(H) \mid \mathfrak a
      \cap S = \emptyset\}$, then $\mathfrak p \in v\text{\rm -spec}(H)$.

\smallskip
\item $v$-$\max (H) \subset v$-$\spec(H)$.
\end{enumerate}
\end{lemma}

\begin{proof}
1. Assume to the contrary that $\mathfrak p \in \mathcal I_v(H)$ is
maximal with respect to $\mathfrak p \cap S = \emptyset$, but
$\mathfrak p$ is not prime. Then there exist elements $a,\,b \in H
\setminus \mathfrak p$ such that $aHb \subset \mathfrak p$. By the
maximal property of $\mathfrak p$, we have $S \cap (\mathfrak p \cup
HaH)_v \ne \emptyset$ and $S \cap (\mathfrak p \cup HbH)_v \ne
\emptyset$. If $s \in S \cap (\mathfrak p \cup HaH)_v$ and $t \in S
\cap (\mathfrak p\cup HbH)_v$, then $sht \in S$ for some $h \in H$,
and using Lemma \ref{3.6}.5 we obtain that
\[
sht \in \bigl( \mathfrak p\cup HaH \bigr)_v \, H \bigl(\mathfrak p
\cup HbH \bigr)_v \subset \bigl[(\mathfrak p\cup HaH) H (\mathfrak
p\cup HbH)\bigr]_v \subset [\mathfrak p \cup HaHbH]_v = \mathfrak
p_v = \mathfrak p\,,
\]
a contradiction.

\smallskip

2. If $\mathfrak m \in v\text{\rm -max}(H)$, then $\mathfrak m \in
\mathcal I_v(H)$ is maximal with respect to $\mathfrak m \cap \{1\}
= \emptyset$, and therefore $\mathfrak m$ is prime by 1.
\end{proof}

\smallskip
Our next step is to introduce completely integrally closed monoids.

\medskip
\begin{lemma} \label{3.9}
Let $H$ be a monoid and $H'$ an overmonoid of $H$.
\begin{enumerate}
\item If $I = (H \DPr H')$, then $H' \subset (I \DPl I)$.

\smallskip
\item Let $a, b \in H$ with $a H' b \subset H$. Then there exists a
      monoid $H''$ with $H \subset H'' \subset H'$ such that $(H \DPr H'')
      \ne \emptyset$ and $(H'' \DPl H') \ne \emptyset$.
\end{enumerate}
\end{lemma}

\begin{proof}
1. Since $H' (H' I) = H' I \subset H$, it follows that $H' I \subset
(H \DPr H') = I$ and hence $H' \subset (I \DPl I)$.

\smallskip
2. We set $H'' = HaH' \cup H$, and obtain that $H \subset H''
\subset H'$, $H'' H'' = H''$, $H''b \subset H$ and $aH' \subset
H''$.
\end{proof}

\medskip
\begin{lemma} \label{3.10}
Let $H$ be a monoid.
\begin{enumerate}
\item The following statements are equivalent{\rm \,:}
      \begin{enumerate}
      \smallskip
      \item[(a)] There is no overmonoid $H'$ of $H$ with $H \subsetneq H' \subset
           \mathsf q (H)$ and $a H' b \subset H$ for some $a, b \in H$.

      \smallskip
      \item[(b)] $(\mathfrak a \DPl \mathfrak a) = (\mathfrak b \DPr \mathfrak b) = H$
                 for all non-empty  left modules $\mathfrak a$ of $H$ which are right $H$-fractional and
                 for all non-empty right modules $\mathfrak b$ of $H$ which are left $H$-fractional.

      \smallskip
      \item[(c)] $(\mathfrak a \DPl \mathfrak a) = (\mathfrak a \DPr \mathfrak a) = H$ for all non-empty   ideals $\mathfrak a$ of $H$.
      \end{enumerate}

\smallskip
\item Suppose that  $H$ satisfies one of the equivalent conditions in
      1. Then  $(H \DPl \mathfrak a) = (H \DPr \mathfrak a)$ and $\big( H \DPl (H \DPr \mathfrak a)\big) = \big( H \DPr (H \DPl
      \mathfrak a)\big)$ for all non-empty fractional ideals $\mathfrak a$ of $H$.
\end{enumerate}
\end{lemma}

\begin{proof}
1. If $H = \mathsf q (H)$, then all statements are fulfilled.
Suppose that $H$ is not a group.

\smallskip
(a) \,$\Rightarrow$\, (b) \ Let $\emptyset \ne \mathfrak a \subset
\mathsf q (H)$ and $a \in H$ with $H \mathfrak a = \mathfrak a$ and
$a \mathfrak a \subset H$. Then $H' = (\mathfrak a \DPl \mathfrak
a)$ is an overmonoid of $H$. If $b \in \mathfrak a \cap H$, then
$aH'b \subset a\mathfrak a \subset H$ and hence $H' = H$ by 1.

\smallskip
(b) \,$\Rightarrow$\, (c) \ Obvious.

\smallskip
(c) \,$\Rightarrow$\, (a) \ Let $H'$ be an overmonoid of $H$ with $a
H' b \subset H$ for some $a, b \in H$. We have to show that $H' =
H$. By Lemma \ref{3.9}.2, there exists a monoid $H''$ with $H
\subset H'' \subset H'$ such that $\mathfrak a = (H \DPr H'') \ne
\emptyset$ and $\mathfrak b = (H'' \DPl H') \ne \emptyset$. Then
Lemma \ref{3.9}.1 implies that $H'' \subset (\mathfrak a \DPl
\mathfrak a) = H$ and $H' \subset (\mathfrak b \DPr \mathfrak b) =
H$.

\smallskip
2. If $\mathfrak a \subset \mathsf q (H)$ is a non-empty fractional ideal,
then Lemma \ref{3.2}.3 and 1. imply that
\[
(H \DPl \mathfrak a) = \big( (\mathfrak a \DPr \mathfrak a) \DPl \mathfrak a \big) = \big( ( \mathfrak a \DPl \mathfrak a) \DPr
\mathfrak a \big) = ( H \DPr \mathfrak a) \,.
\]
Since $(H \DPl \mathfrak a) = (H \DPr \mathfrak a)$ is a non-empty fractional ideal, the
previous argument implies that \newline $\big( H \DPl (H \DPr
\mathfrak a)\big) = \big( H \DPr (H \DPl \mathfrak a)\big)$.
\end{proof}

\medskip
\begin{definition} \label{3.11}
A monoid $H$ is said to be
\begin{itemize}
\item {\it completely integrally closed} \ if it satisfies the
equivalent conditions of Lemma \ref{3.10}.1.

\smallskip
\item  {\it $v$-noetherian} \ if it
satisfies the ascending chain condition on $v$-ideals of $H$.

\smallskip
\item a {\it Krull monoid} \ if it is completely integrally closed and $v$-noetherian.
\end{itemize}
\end{definition}

\smallskip
If $H$ is a commutative monoid, then the above notion of being
completely integrally closed coincides with the usual one (see
\cite[Section 2.3]{Ge-HK06a}). We need a few notions from the theory
of  po-groups (we follow the terminology of \cite{St10a}).
Let $Q = (Q, \cdot)$ be a multiplicatively written group with unit
element $1 \in Q$, and let $\le$ be a partial order on $Q$.  Then
$(Q, \cdot, \le)$ is said to be
\begin{itemize}
\item a {\it po-group} \ if $x \le y$ implies that $axb \le ayb$ for
      all $x,y,a,b \in Q$.

\item {\it directed} \ if each two element subset of $Q$ has an upper and a lower bound.

\item {\it integrally closed} \ if for all $a, b \in Q$, $a^n \le b$
for all $n \in \N$ implies that $a \le 1$.
\end{itemize}

\medskip
\begin{proposition} \label{3.12}
Let $H$ be a completely integrally closed monoid.
\begin{enumerate}
\item Every non-empty fractional $v$-ideal is $v$-invertible, and $v$-$\max (H) = v$-$\spec (H) \setminus \{\emptyset\}$.

\smallskip
\item Equipped with the set-theoretical inclusion as a partial
order and $v$-multiplication as group operation, the group $\mathcal
F_v (H)^{\times}$ is a directed integrally closed po-group.

\smallskip
\item $\mathcal I_v^* (H)$ is a commutative monoid with quotient
      group $\mathcal F_v (H)^{\times}$.

\smallskip
\item
If $\mathfrak a,\, \mathfrak b \in \mathcal I^*_v (H)$, then
$\mathfrak a \supset \mathfrak b$ if and only if \ $\mathfrak a \t
\mathfrak b$ in $\mathcal I^*_v(H)$. In particular, $(\mathfrak a
\cup \mathfrak b)_v = \gcd (\mathfrak a , \mathfrak b)$ in $\mathcal
I_v^*(H)$, and $\mathcal I_v^* (H)$ is reduced.
\end{enumerate}
\end{proposition}

\begin{proof}
1. Let $\emptyset \ne \mathfrak a \in \mathcal F_v (H)$. Using Lemma
\ref{3.6}.2. and that $H$ is completely integrally closed, we obtain
that $(\mathfrak a \mathfrak a^{-1})_v = (\mathfrak a_v \DP
\mathfrak a)^{-1} = ( \mathfrak a \DP \mathfrak a)^{-1} = H^{-1} =
H$. Since $\mathfrak a^{-1} \in \mathcal F_v (H)$, we may apply this
relation for $\mathfrak a^{-1}$ and get $(\mathfrak a^{-1} \mathfrak
a)_v = H$. Therefore it follows that
\[
\mathfrak a \cdv \mathfrak a^{-1} = (\mathfrak a \mathfrak a^{-1})_v
=  H  = (\mathfrak a^{-1} \mathfrak a)_v = \mathfrak a^{-1} \cdv
\mathfrak a \,.
\]

By Lemma \ref{3.8}.2, we have $v$-$\max (H) \subset v$-$\spec (H)
\setminus \{\emptyset\}$. Assume to the contrary that there are
$\mathfrak p, \mathfrak q \in v$-$\spec (H)$ with $\emptyset \ne
\mathfrak p \subsetneq \mathfrak q \subset H$. Since $\mathfrak q$
is $v$-invertible, we get $\mathfrak p = \mathfrak q \cdv \mathfrak
a$ with $\mathfrak a = \mathfrak q^{-1} \cdv \mathfrak p \subset H$.
Since $\mathfrak p$ is a prime ideal and $\mathfrak q \not\subset
\mathfrak p$, it follows that $\mathfrak a \subset \mathfrak p$.
Then $\mathfrak a = \mathfrak b \cdv \mathfrak p$ with $\mathfrak b
= \mathfrak a \cdv \mathfrak p^{-1} \subset H$, whence $\mathfrak p
= \mathfrak q \cdv \mathfrak b \cdv \mathfrak p$ and thus $H =
\mathfrak q \cdv \mathfrak b$, a contradiction.

\smallskip
2. Clearly, $(\mathcal F_v(H)^{\times}, \, \cdv \ , \subset)$ is a
po-group. In order to show that it is directed, consider $\mathfrak
a, \mathfrak b \in \mathcal F_v (H)^{\times}$. Then $\mathfrak a
\cdv \mathfrak b \in \mathcal F_v^{\times} (H)$ is a lower bound of
$\{\mathfrak a, \mathfrak b\}$, and $(\mathfrak a \cup \mathfrak
b)_v$ is an upper bound. In order to show that it is integrally
closed, let $\mathfrak a, \mathfrak b \in \mathcal F_v (H)^{\times}$
be given such that $\mathfrak a^n \subset \mathfrak b$ for all $n
\in \mathbb N$. We have to show that $\mathfrak a \subset H$. The
set
\[
\mathfrak a_0 = \bigcup_{n \ge 1} \mathfrak a^n \ \subset \ \mathfrak b
\]
is a non-empty fractional ideal, and we get $\mathfrak a \subset (
\mathfrak a_0 \DPl \mathfrak a_0) = H$, since $H$ is completely
integrally closed.

\smallskip
3. Since $(\mathcal F_v(H)^{\times}, \, \cdv \ , \subset)$ is a
directed integrally closed po-group by 2., $\mathcal
F_v(H)^{\times}$ is a commutative group by \cite[Theorem
2.3.9]{St10a}. Since $\mathcal I_v^* (H) = \mathcal F_v (H)^{\times}
\cap \mathcal I_v (H)$ by Lemma \ref{3.6}.6, it follows that
$\mathcal I_v^* (H)$ is a commutative monoid. In order to show that
$\mathcal F_v (H)^{\times}$ is a quotient group of $\mathcal I_v^*
(H)$, let $\mathfrak c \in \mathcal F_v (H)^{\times}$ be given. We
have to find some $\mathfrak a \in \mathcal I_v^* (H)$ such that
$\mathfrak a \cdv \mathfrak c \in \mathcal I_v^* (H)$, and for that
it suffices to verify that $\mathfrak a \cdv \mathfrak c \subset H$.
Now, since $\mathfrak c$ is a fractional ideal, there exists some $c
\in H$ such that $c\mathfrak c \subset H$, thus $(HcH)_v \in
\mathcal I_v^* (H)$ and, by Lemma \ref{3.6}.5,
\[
(HcH)_v \cdv \mathfrak c = \big( (HcH)_v \mathfrak c \big)_v =
(Hc\mathfrak c)_v \subset H_v = H \,.
\]

\smallskip
4. Note that $\mathcal I_v^* (H)$ is commutative by 3., and hence
the  greatest common divisor is formed in a commutative monoid. Thus
the in particular statements follow immediately from the main
statement. In order to show that divisibility is equivalent to
containment, we argue as before. Let $\mathfrak a, \mathfrak b \in
\mathcal I_v^* (H)$. If $\mathfrak a \t \mathfrak b$ in $\mathcal
I^*_v (H)$, then $\mathfrak b = \mathfrak a \cdv \mathfrak c$ for
some $\mathfrak c \in \mathcal I_v^*(H)$, and therefore $\mathfrak b
\subset \mathfrak a$. If $\mathfrak b \subset \mathfrak a$, then
$\mathfrak b \cdv \mathfrak a^{-1} \subset \mathfrak a \cdv
\mathfrak a^{-1} = H$, and thus $\mathfrak b \cdv \mathfrak a^{-1}
\in  \mathcal F_v (H)^{\times} \cap \mathcal I_v (H) = \mathcal
I^*_v(H)$. The relation $\mathfrak b = (\mathfrak b \cdv \mathfrak
a^{-1}) \cdv \mathfrak a$ shows that $\mathfrak a \t \mathfrak b$ in
$\mathcal I_v^*(H)$.
\end{proof}

\smallskip
The missing part are ideal theoretic properties of $v$-noetherian monoids.

\medskip
\begin{proposition} \label{3.13}
Suppose that $(H \DPl \mathfrak c) = (H \DPr \mathfrak c)$ for all fractional ideals $\mathfrak c$ of $H$.
\begin{enumerate}
\item The following statements are equivalent{\rm \,:}
\begin{enumerate}
\smallskip
\item[(a)] $H$ is $v$-noetherian.

\smallskip

\item[(b)] Every non-empty set of $v$-ideals of $H$ has a maximal
element {\rm (}with respect to the inclusion{\rm )}.

\smallskip

\item[(c)] Every non-empty set of fractional $v$-ideals of $H$ with
non-empty intersection has a minimal element {\rm (}with respect to
the inclusion{\rm )}.

\smallskip

\item[(d)]
For every non-empty ideal $\mathfrak a \subset H$, there exists a
finite subset $E \subset \mathfrak a$ such that $(HEH)^{-1} =
\mathfrak a^{-1}$.
\end{enumerate}

\smallskip
\item If $H$ is $v$-noetherian and $\mathfrak a \in \mathcal I_v^* (H)$, then there exists a
      finite set $E \subset \mathfrak a$ such that $\mathfrak a =
      (HEH)_v$.

\smallskip
\item If $H$ is $v$-noetherian and  $a \in H$, then the set $\{ \mathfrak p \in v\text{\rm
-spec}(H) \mid a \in \mathfrak p\}$ is finite.
\end{enumerate}
\end{proposition}

\begin{proof}
1. (a)\, $\Rightarrow$\, (b) \ If $\emptyset \ne \Omega \subset
\mathcal I_v(H)$ has no maximal element, then every $\mathfrak a \in
\Omega$ is properly contained in some $\mathfrak a' \in \Omega$. If
$\mathfrak a_0 \in \Omega$ is arbitrary and the sequence $(\mathfrak
a_n)_{n \ge 0}$ is recursively defined by $\mathfrak a_{n+1} =
\mathfrak a'_n$ for all $n\ge 0$, then $(\mathfrak a_n)_{n\ge 0}$ is
an ascending sequence of $v$-ideals not becoming stationary.

\smallskip

(b)\, $\Rightarrow$\, (c) \ Suppose that $\emptyset \ne \Omega
\subset \mathcal F_v(H)$ and $a\in \mathfrak a$ for all $\mathfrak a
\in \Omega$. Then the set $\Omega^* = \{a \mathfrak a^{-1} \mid
\mathfrak a \in \Omega\} \subset \mathcal I_v(H)$ has a maximal
element $a \mathfrak a_0^{-1}$ with $\mathfrak a_0 \in \Omega$, and
then $\mathfrak a_0$ is a minimal element of $\Omega$.

\smallskip

(c)\, $\Rightarrow$\, (d) \ If $\emptyset \ne E \subset \mathfrak
a$, then $\emptyset \ne \mathfrak a^{-1} \subset (HEH)^{-1} \in
\mathcal F_v(H)$. Thus the set $\Omega = \{(HEH)^{-1} \mid \emptyset
\ne E\subset \mathfrak a,\, E \text { finite}\}$ has a minimal
element $(HE_0H)^{-1}$, where $E_0 \subset \mathfrak a$ is a finite
non-empty subset. Then $(HE_0H)^{-1} \supset \mathfrak a^{-1}$, and
we assert that equality holds. Assume to the contrary that there
exists some $u\in (HE_0H)^{-1} \setminus \mathfrak a^{-1}$. Then
there exists an element $a\in \mathfrak a$ such that $ua\notin H$,
and if $E_1 = E_0 \cup \{a\}$, then $u \notin (HE_1H)^{-1}$ and
consequently $(HE_1H)^{-1} \subsetneq (HE_0H)^{-1}$, a
contradiction.

\smallskip

(d)\, $\Rightarrow$\, (a) \ Let $\mathfrak a_1 \subset \mathfrak
a_2\subset \dots$ be an ascending sequence of $v$-ideals. Then
\[
\mathfrak a  = \bigcup_{n\ge 1}\mathfrak a_n \subset H
\]
is an ideal of $H$, and we pick a finite non-empty subset $E \subset
\mathfrak a$  such that $(HEH)^{-1} = \mathfrak a^{-1}$. Then there
exists some $m\ge 0$ such that $E \subset \mathfrak a_m$. For all $n
\ge m$ we obtain $\mathfrak a_n \subset \mathfrak a \subset
\mathfrak a_v = (HEH)_v \subset \mathfrak a _m $ and hence
$\mathfrak a_n = \mathfrak a_m$.

\smallskip
2. Let $H$ be a $v$-noetherian and $\mathfrak a \in \mathcal I_v^* (H)$. By 1., there exists a
finite subset $E \subset \mathfrak a$ such that $(HEH)^{-1} =
\mathfrak a^{-1}$ and therefore $(HEH)_v = \mathfrak a_v = \mathfrak
a$.

\smallskip
3. Assume to the contrary that $H$ is  $v$-noetherian and that there exists some $a \in H$ such that
the set $\Omega = \{\mathfrak p \in v\text{\rm -spec}(H) \mid a \in
\mathfrak p\}$ is infinite. Then 1. implies that there is a sequence
$(\mathfrak p_n)_{n\ge 0}$ in $\Omega$ such that, for all $n\ge 0,\
\mathfrak p_n$ is maximal in $\Omega\setminus \{\mathfrak
p_0,\dots,\mathfrak p_{n-1}\}$, and again by 1., the set
$\{\mathfrak p_0 \cap \mathfrak p_1 \cap \ldots \cap \mathfrak p_n
\mid n \in \N_0\}$ has a minimal element. Hence there exists some $n
\in \N_0$ such that $\mathfrak p_0 \cap \dots \cap \mathfrak p_n =
\mathfrak p_0 \cap \dots \cap \mathfrak p_{n+1} \subset \mathfrak
p_{n+1}$. Since $\mathfrak p_{n+1}$ is a prime ideal, Lemma
\ref{3.7} implies that there exists some $i\in [0,n]$ such that
$\mathfrak p_i \subset \mathfrak p_{n+1}$. Since now $\mathfrak
p_{n+1}\in \Omega \setminus \{ \mathfrak p_1, \ldots, \mathfrak
p_n\} \subset \Omega \setminus \{ \mathfrak p_1, \ldots, \mathfrak
p_{i-1} \}$ and $\mathfrak p_i$ is maximal in the larger set, it
follows that $\mathfrak p_{n+1} \subset \mathfrak p_i$, and hence
$\mathfrak p_{n+1} = \mathfrak p_i \in  \Omega \setminus \{\mathfrak
p_1, \ldots, \mathfrak p_n\}$, a contradiction.
\end{proof}

\medskip
In contrast to the commutative setting the set $\{ \mathfrak p \in v\text{\rm
-spec}(H) \mid a \in \mathfrak p\}$ can be empty. We will provide an example in Section \ref{5} after having established the relationship between Krull monoids and Krull rings (see Example \ref{5.2}).

\medskip
\begin{theorem}[\bf Ideal theory of Krull monoids] \label{3.14}~
Let $H$ be a Krull monoid. Then $\mathcal I_v^* (H)$ is a free
abelian monoid with basis $v$-$\max (H) = v$-$\spec (H) \setminus
\{\emptyset\}$.
\end{theorem}

\begin{proof}
Since $H$ is $v$-noetherian and since divisibility in $\mathcal
I_v^* (H)$ is equivalent to containment (by Proposition
\ref{3.12}.4), $\mathcal I_v^* (H)$ is reduced and satisfies the
divisor chain condition. Therefore, it is atomic by
\cite[Proposition 1.1.4]{Ge-HK06a}. Again by the equivalence of
divisibility and containment, the set of atoms of $\mathcal I_v^*
(H)$ equals $v$-$\max (H)$, and by Proposition \ref{3.12}, we have
$v$-$\max (H) = v$-$\spec (H) \setminus \{\emptyset\}$. Since every
non-empty prime $v$-ideal is a prime element of $\mathcal I_v^*
(H)$, every atom of $\mathcal I_v^* (H)$ is a prime element,  and
thus $\mathcal I_v^* (H)$ is a free abelian monoid with basis
$v$-$\max (H)$ by \cite[1.1.10 and 1.2.2]{Ge-HK06a}.
\end{proof}

\bigskip
\section{Divisor homomorphisms and normalizing monoids}
\label{4}
\bigskip

The classic concept of a divisor theory was first presented in an
abstract (commutative) setting by  Skula \cite{Sk70}, and after
that it was generalized in many steps (see e.g. \cite{Ge-HK94}, and
the presentations in \cite{HK98, Ge-HK06a}). In this section we
investigate divisor homomorphisms and divisor theories in a
non-commutative setting. We study normal elements and normalizing
submonoids of rings and monoids as introduced by  Wauters
\cite{Wa84a} and  Cohn \cite[Section 3.1]{Co85a}. For the role
of normal elements in ring theory see \cite[Chapter 12]{Go-Wa04a} and \cite[Chapter 10]{Mc-Ro01a}.
The normalizing monoid $\mathsf N (H)$ of a monoid $H$ plays a
crucial role in the study of semigroup algebras $K[H]$ (see \cite{Je-Ok07a}).
In this context, Jespers and Okni{\'n}ski showed that completely integrally closed monoids, whose quotient groups are finitely generated
torsion-free nilpotent groups and which satisfy the ascending chain condition on right ideals, are normalizing (see \cite[Theorem 2]{Je-Ok99a}).
Recall that, if $R$ is a prime ring and $a \in R
\setminus \{0\}$ is a normal element, then $a$ is a regular element.
The main results in this section are the divisor theoretic
characterization of normalizing Krull monoids together with its
consequences (Theorem \ref{4.13} and Corollary \ref{4.14}).

\medskip
\begin{definition} \label{4.1}~
\begin{enumerate}
\item A homomorphism of monoids  $\varphi \colon H \to D$ is called a
      \begin{itemize}
      \smallskip
      \item (left and right) {\it divisor homomorphism} if  $\varphi (u) \tl \varphi (v)$ implies that $u \tl
            v$ and $\varphi (u) \tr \varphi (v)$ implies that $u \tr v$
            for all $u, v \in H$.

      \smallskip
      \item (left and right) {\it cofinal} if for every $a \in D$ there exist
      $u, v \in H$ such that $a \tl \varphi (u)$ and $a \tr \varphi
      (v)$ (equivalently, $aD \cap \varphi (H) \ne \emptyset$ and
      $Da \cap \varphi (H) \ne \emptyset$).
      \end{itemize}

\smallskip
\item A  \ {\it divisor theory} \ (for $H$) \ is a divisor
      homomorphism $\varphi \colon H \to D$ such that $D = \mathcal F(P)$
      for some set $P$ and, for every $p \in P$, there exists a finite
      subset $\emptyset \ne X \subset H$ satisfying $p = \gcd \bigl(
      \varphi(X) \bigr)$.

\smallskip
\item A submonoid $H \subset D$ is called
      \begin{itemize}
      \smallskip
      \item {\it cofinal} if the embedding $H \hookrightarrow D$ is
            cofinal.

      \smallskip
      \item {\it saturated} if the embedding  $H \hookrightarrow D$
            is a divisor homomorphism.
      \end{itemize}
\end{enumerate}
\end{definition}

\medskip
\begin{definition} \label{4.2}
Let $H$ be a cancellative semigroup.
\begin{enumerate}
\item  An element $a \in H$ is said to be {\it normal} (or {\it invariant}) if $aH = Ha$. The subset $\mathsf N (H) = \{ a \in H \mid  aH = Ha \} \subset H$ is called the {\it normalizing submonoid} (or {\it invariant submonoid}) of $H$,
       and $H$ is said to be {\it normalizing} if $\mathsf N (H) =
       H$ (Lemma \ref{4.3} will show that $\mathsf N (H)$ is indeed a normalizing submonoid).

\smallskip
\item An element $a \in H$ is said to be {\it weakly normal} if $aH^{\times} = H^{\times}a$. The subset $H^{\mathsf w} = \{ a \in H \mid aH^{\times} = H^{\times}a \} \subset H$ is called the {\it weakly normal submonoid} of $H$, and $H$ is said to be {\it weakly normal} if $H^{\mathsf w} = H$.

\smallskip
\item  Two elements $a, b \in H$ are said to be {\it associated} if $a \in H^{\times} b H^{\times}$ (we write $a \simeq b$, and note that this is an equivalence relation on $H$).

\smallskip
\item We denote  by $\mathcal P (H) = \{aH \mid a \in H\}$ the set of principal right ideals,  by $\mathcal P^{\mathsf n} (H) =
      \{ aH \mid a \in \mathsf N (H)\}$ the set of normalizing
      principal ideals, by $\mathsf C (H) = \{ a \in H \mid ab = ba \ \text{for all} \ b \in H\}$ the {\it center} of $H$, and we set  $H_{\red} = \{ aH^{\times} \mid a \in H^{\mathsf w} \}$,
\end{enumerate}
\end{definition}

\medskip
\begin{lemma} \label{4.3}
Let $H$ be a cancellative semigroup.
\begin{enumerate}
\item If $H$ is normalizing, then $H$ is a monoid.

\smallskip
\item $\mathsf N (H)$ is a subsemigroup with $H^{\times} \subset \mathsf N (H)$, and if $H$ is a monoid, then $\mathsf N (H)
      \subset H$ is a normalizing saturated submonoid.

\smallskip
\item $\mathsf C (H) \subset \mathsf N (H)$ is a commutative saturated submonoid.
\end{enumerate}
\end{lemma}

\begin{proof}
1. Let $H$ be a normalizing semigroup. If $a, b \in H$, then $ab \in
aH = Ha$ implies the existence of an element $c \in H$ such that $ab
= ca$ and hence $Ha \cap Hb \ne \emptyset$. Similarly, we get that
$aH \cap bH \ne \emptyset$. Thus the left and right Ore condition is
satisfied, and $H$ is a monoid.

\smallskip
2. If $a, b \in H$ with $aH = Ha$ and $bH = Hb$, then $abH = aHb =
Hab$. Since $1 \in \mathsf N (H)$, it follows that $\mathsf N (H)
\subset H$ is a subsemigroup. Since $\varepsilon H = H = H
\varepsilon$ for all $\varepsilon \in H^{\times}$, we have
$H^{\times} \subset \mathsf N (H)$.

Suppose that $H$ is a monoid. In order to show that $\mathsf N (H)$
is normalizing, we have to verify that $a \mathsf N (H) = \mathsf N
(H) a$ for all $a \in \mathsf N (H)$. Let $a, b \in \mathsf N (H)$.
Since $ab \in aH = Ha$, there exists some $c \in H$ such that $ab =
ca$. Since $H$ is a monoid, $a \in H$ is invertible in $\mathsf q
(H)$, and we get $cH = aba^{-1}H = H aba^{-1} = Hc$, which shows
shows that $c \in \mathsf N (H)$. This implies that $a \mathsf N(H)
\subset \mathsf N (H) a$, and by repeating the argument we obtain
equality.

In order to show that $\mathsf N (H) \subset H$ is saturated, let
$a, b \in \mathsf N (H)$ be given such that $a \tl b$ in $H$. Then
there exists an element $c \in H$ such that $b = ac$. Since $cH =
a^{-1}bH = H a^{-1}b = Hc$, it follows that $c \in \mathsf N (H)$,
and hence $a \tl b$ in $\mathsf N (H)$. If $a, b \in \mathsf N (H)$
such that $a \tr b$ in $H$, then we similarly get that $a \tr b$ in
$\mathsf N (H)$. Thus $\mathsf N (H) \subset H$ is a saturated
submonoid.

\smallskip
3. It follows by the definition that $\mathsf C (H) \subset \mathsf N (H)$ is a commutative submonoid.
In order to show that $\mathsf C (H) \subset \mathsf N (H)$ is saturated, let
$a, b \in \mathsf C (H)$ be given such that $a \tl b$ in $\mathsf N (H)$. Then
there exists an element $c \in \mathsf N (H)$ such that $b = ac$. For every $d \in H$, we have
$c d = a^{-1}bd = d a^{-1}b = dc$, hence $c \in \mathsf C (H)$ and $a \tl b$ in $\mathsf C (H)$.
We argue similarly in case of right divisibility and obtain that $\mathsf C (H) \subset \mathsf N (H)$ is saturated.
\end{proof}

\medskip
\begin{lemma} \label{4.4}
Let $H$ be a monoid.
\begin{enumerate}
\item $H^{\mathsf w}$ is a monoid with $H^{\times} \subset \mathsf N (H) \subset H^{\mathsf w} \subset H$.
      To be associated is a congruence relation on $H^{\mathsf w}$, and $[a]_{\simeq} =
      aH^{\times} = H^{\times}a$ for all $a \in H^{\mathsf w}$.

\smallskip
\item The quotient semigroup $H^{\mathsf w} /\negthinspace \simeq \ = \ H_{\red}$ is a monoid with quotient group $\mathsf q (H^{\mathsf w})/H^{\times}$.
      Moreover, $H$ is normalizing if and only if $H = H^{\mathsf w}$ and $H_{\red}$ is
      normalizing.

\smallskip
\item Let $D$ be a  monoid and $\varphi \colon H \to D$ a monoid homomorphism.
      Then there exists a unique homomorphism $\varphi_{\red} \colon H_{\red} \to D_{\red}$ satisfying $\varphi_{\red} (a H^{\times}) =
      \varphi (a) D^{\times}$ for all $a \in H^{\mathsf w}$.

\smallskip
\item The map $f \colon \mathcal I_s (H^{\mathsf w}) \to \mathcal
      I_s(H_{\red})$, $I \mapsto \overline I = \{ uH^{\times} \mid u \in I\}$ is an inclusion preserving
      bijection. Moreover, $I$ is a principal right ideal or a divisorial ideal if and
      only if $\overline I$ has the same property.
\end{enumerate}
\end{lemma}

\begin{proof}
1. If $a, b \in H$ are weakly normal, then $abH^{\times} =
aH^{\times}b = H^{\times} ab$, and hence $ab$ is weakly normal. Next
we show that every normal element is weakly normal. Let $a \in H$ be
normal. If $\varepsilon \in H^{\times}$, then $a \varepsilon = ba
\in aH = Ha$ with $b \in H$ and hence $a \varepsilon a^{-1} \in H$.
Similarly, we get $a \varepsilon^{-1}a^{-1} \in H$, hence $a
\varepsilon a^{-1} \in H^{\times}$, and $a \varepsilon = (a
\varepsilon a^{-1})a \in H^{\times} a$. This shows that $aH^{\times}
\subset H^{\times}a$, and by symmetry we get $aH^{\times} =
H^{\times} a$.

By Lemma \ref{4.3}, we infer that $H^{\mathsf w}$ is a monoid with
$H^{\times} \subset \mathsf N (H) \subset H^{\mathsf w} \subset H$.
Clearly, $\simeq$ is a congruence relation on $H^{\mathsf w}$ and
$[a]_{\simeq} = aH^{\times} = H^{\times}a$ for all $a \in H^{\mathsf
w}$.

\smallskip
2. The group $\mathsf q (H^{\mathsf w})/H^{\times}$ is a quotient
group of $H_{\red}$, and hence $H_{\red}$ is a monoid.

Suppose that $H$ is normalizing. Then $\mathsf N (H) \subset
H^{\mathsf w} \subset H = \mathsf N (H)$, and we verify that
$H_{\red}$ is normalizing. Since
\[
\{ a c \mid c \in H \} = aH = Ha = \{ ca \mid c \in H \} \,,
\]
it follows that
\[
\begin{aligned}
(aH^{\times}) H_{\red} & = \{ aH^{\times} c H^{\times} \mid c \in H
\} = \{ ac H^{\times} \mid c \in H \} = \{ ca H^{\times} \mid c \in
H \} \\ & = \{ cH^{\times} a H^{\times} \mid c \in H \} =
H_{\red}(aH^{\times}) \,,
\end{aligned}
\]
and thus $H_{\red}$ is normalizing.

\smallskip
Conversely, suppose that $H = H^{\mathsf w}$ and that $H_{\red}$ is
normalizing. Let $a \in H$. By symmetry it suffices to verify that
$a H \subset Ha$. Let $c \in H$. Since
\[
acH^{\times} \in \{ (aH^{\times})(d H^{\times}) = a d H^{\times}
\mid d \in H \} = \{ (dH^{\times})(aH^{\times}) = daH^{\times} \mid
d \in H \} \,,
\]
there exist $d \in H$  and $\varepsilon \in H^{\times}$ such that
$ac = da \varepsilon$. Since $aH^{\times} = H^{\times}a$, there is
an $\eta \in H^{\times}$ such that $a \varepsilon = \eta a$, and
hence $ac = (d\eta)a \in Ha$.

\smallskip
3. If $b, c \in H^{\mathsf w}$ with $bH^{\times} = c H^{\times}$, then $\varphi (b)D^{\times} = \varphi (c) D^{\times}$. Hence we can define a map $\varphi_{\red} \colon H_{\red} \to D_{\red}$ satisfying $\varphi_{\red} (aH^{\times}) = \varphi (a) D^{\times}$. Obviously, $\varphi_{\red}$ is uniquely determined and a homomorphism.

\smallskip
4. We define a map $g \colon \mathcal
I_s(H_{\red}) \to \mathcal I_s (H^{\mathsf w})$ by setting $g (J) = \{ v \in H^{\mathsf w} \mid v H^{\times} \in J\}$ for all $J \in \mathcal
I_s(H_{\red})$. Obviously, $f$ and $g$ are inclusion preserving, inverse to each other, and hence $f$ is bijective.

If $I = aH^{\mathsf w}$, then $f (I) = \{abH^{\times} = (aH^{\times})(bH^{\times}) \mid b  \in H^{\mathsf w} \} = (aH^{\times}) H_{\red}$, and if $J = (aH^{\times}) H_{\red}$,
      then $g (J) = a H^{\mathsf w}$.

If $A \subset \mathsf q (H^{\mathsf w})$, then
\[
\begin{aligned}
(H^{\mathsf w} \DPl A) H^{\times} & = \{ uH^{\times} \mid u \in \mathsf q (H^{\mathsf w}), \, uA \subset H^{\mathsf w} \}
  = \{ uH^{\times} \mid u \in \mathsf q (H^{\mathsf w}), \, u\{a H^{\times} \mid a \in A\} \subset H_{\red} \} \\ & = \big( H_{\red} \DPl \{a H^{\times} \mid a \in A \} \big) \,.
\end{aligned}
\]
The analogous statement is true for right quotients, and thus the assertion for divisorial ideals follows.
\end{proof}

\medskip
\begin{lemma} \label{4.5}
Let $H$ be a  monoid. Then the following statements are
equivalent{\rm \,:}
\begin{enumerate}
      \smallskip
      \item[(a)] $H$ is normalizing.

      \smallskip
      \item[(b)] For all $X \subset \mathsf q (H)$, $(H \DPl X) = (H \DPr X)$.

      \smallskip
      \item[(c)] For all $X \subset \mathsf q (H)$, $HX = XH$.

      \smallskip
      \item[(d)] Every $($fractional$)$ left ideal is a $($fractional$)$ ideal.

      \smallskip
      \item[(e)] Every  divisorial $($fractional$)$ left ideal is a divisorial  $($fractional$)$  ideal.

      \smallskip
      \item[(f)] For every $a \in \mathsf q (H)$, $Ha$ is a fractional ideal.
\end{enumerate}
\end{lemma}

\smallskip
\noindent {\it Remark.} Of course, the  statements on right ideals,
symmetric to (d), (e) and (f), are also equivalent.

\begin{proof}
(a) \,$\Rightarrow$\, (b) \ If $X \subset \mathsf q (H)$, then
\[
(H \DPl X) = \bigcap_{a \in X} (H \DPl a) = \bigcap_{a \in X} (H
\DPl aH) = \bigcap_{a \in X} Ha^{-1} = \bigcap_{a \in X} a^{-1}H =
(H \DPr X) \,.
\]

\smallskip
(b) \,$\Rightarrow$\, (c) \ If $X \subset \mathsf q (H)$, then
\[
HX = \bigcup_{a \in X} Ha = \bigcup_{a \in X} (H \DPl a^{-1}H) =
\bigcup_{a \in X} (H \DPl a^{-1}) = \bigcup_{a \in X} (H \DPr
a^{-1})  = \bigcup_{a \in X} aH = XH \,.
\]

\smallskip
(c) \,$\Rightarrow$\, (d) \,$\Rightarrow$\, (e) \,$\Rightarrow$\,
(f) \ Obvious.

\smallskip
(f) \,$\Rightarrow$\, (a) \ Let $a \in H$. Then $Ha = HaH \supset
aH$, $Ha^{-1} = Ha^{-1}H \supset a^{-1}H$ and hence $aH \supset Ha$,
which implies that $aH = Ha$.
\end{proof}

\medskip
\begin{lemma} \label{4.6}
Let $H$ be a weakly normal monoid, $\pi \colon H \to H_{\red}$ the
canonical epimorphism, and let $\varphi \colon H \to D$ be a
homomorphism to a monoid $D$.
\begin{enumerate}
\item If $\varphi$ is a divisor homomorphism and $\psi \colon D \to D'$ is a divisor homomorphism to a  monoid $D'$,
      then $\psi \circ \varphi \colon H \to D'$ is a divisor homomorphism.

\smallskip
\item $\pi$ is a cofinal divisor homomorphism, and $\varphi$ is a divisor homomorphism if and only if $\varphi_{\red} \colon H_{\red} \to D_{\red}$ is a divisor
      homomorphism. If $\varphi$ is a divisor homomorphism, then
      $\varphi_{\red}$ is injective, $H_{\red} \cong \varphi_{\red}
      (H_{\red})$ and $\varphi_{\red} (H_{\red}) \subset D_{\red}$ is a
      saturated submonoid.

\smallskip
\item If $D = \mathcal F (P)$, then $\varphi$ is a divisor theory if and only if $\varphi_{\red} \colon H_{\red} \to D$ is a divisor theory.
\end{enumerate}
\end{lemma}

\begin{proof}
1. Suppose that $\varphi$ and $\psi$ are divisor homomorphisms, and
let $a, b \in H$ such that $\psi \bigl ( \varphi (a) \bigr) \tl \psi
\bigl ( \varphi (b) \bigr)$. Since $\psi$ is a divisor homomorphism,
we infer that $\varphi (a) \tl \varphi (b)$, and since $\varphi $ is
a divisor homomorphism, we obtain that $a \tl b$. The analogous
argument works for right divisibility.

\smallskip
2. The first statements are clear. Now suppose that $\varphi$ is a
divisor homomorphism, and let $a, b \in H$ with $\varphi (a) =
\varphi (b)$. Then $\varphi (a) \t \varphi (b)$, $\varphi (b) \t
\varphi (a)$, hence $a \t b$, $b \t a$, and thus $a H^{\times} = b
H^{\times}$. Thus $\varphi_{\red}$ is injective, $H_{\red} \cong
\varphi_{\red}
      (H_{\red})$, and since $\varphi_{\red}$ is a divisor
      homomorphism, $\varphi_{\red}(H_{\red}) \subset D_{\red}$ is
      saturated.

\smallskip
3. By 2., it remains to verify that $\varphi$ satisfies the
condition involving the greatest common divisor if and only if
$\varphi_{\red}$ does. Indeed, if $a_1, \ldots, a_n \in H$, then
$\varphi_{\red} (a_i H^{\times}) = \varphi (a_i)$ for all $i \in [1,
n]$ and hence
\[
\gcd \big( \varphi (a_1), \ldots, \varphi (a_n) \big) = \gcd \big(
\varphi_{\red} (a_1 H^{\times}), \ldots \varphi_{\red}
(a_nH^{\times}) \big) \,,
\]
which implies the assertion.
\end{proof}

\medskip
\begin{lemma} \label{4.7}
Let $H$ be a  monoid.
\begin{enumerate}
\item If $a, b \in \mathsf N (H)$, then $aH$, $bH$ are divisorial ideals of $H$, and
$(aH) \cdv (bH) = (aH)(bH) = abH$. Thus the usual ideal
multiplication coincides with the $v$-multiplication.

\smallskip
\item Equipped with usual ideal multiplication,  $\mathcal P^{\mathsf n} (H)$ is a normalizing
monoid. It is a saturated submonoid of $\mathcal I_v^* (H)$, and the
inclusion is cofinal if and only if $\mathfrak a \cap \mathsf N (H)
\ne \emptyset$ for all $\mathfrak a \in \mathcal I_v^* (H)$.

\smallskip
\item The map $f \colon \mathsf N (H)_{\red} \to \mathcal P^{\mathsf n} (H)$, defined by
      $ a H^{\times} = a \mathsf N (H)^{\times} \mapsto aH$ for all $a \in \mathsf N (H)$,
      is an isomorphism.

\smallskip
\item If $H$ is normalizing, then the map \ $\partial \colon H \to \mathcal I_v^* (H)$, defined by
      $\partial (a) = aH$ for all $a \in H$, is a cofinal divisor
      homomorphism.
\end{enumerate}
\end{lemma}

\begin{proof}
1. If $c \in \mathsf N (H)$, then $cH$ is an ideal of $H$ by
definition, and it is divisorial by Lemma \ref{3.4}.4. If $a, b \in
\mathsf N (H)$, then
\[
(aH) \cdv (bH) = \big( (aH)(bH) \big)_v = (abH)_v = abH \,.
\]

\smallskip
2. and 3. Let $a, b \in H$. Since $aH = bH$ if and only if
$aH^{\times} = bH^{\times}$, $f$ is injective, and obviously $f$ is
a semigroup epimorphism. Since $\mathsf N (H)$ is normalizing by
Lemma \ref{4.3}, its associated reduced monoid $\mathsf N
(H)_{\red}$ is normalizing, and thus $\mathcal P^{\mathsf n} (H)$ is
a normalizing monoid. By 1., it is a submonoid of $\mathcal I_v^*
(H)$.

In order to show that $\mathcal P^{\mathsf n} (H) \subset \mathcal
I_v^* (H)$ is saturated, let $a, b \in \mathsf N (H)$ such that $aH
\tl bH$ in $\mathcal I_v^* (H)$. Then there exists some $\mathfrak a
\in \mathcal I_v^* (H)$ such that $bH = aH \cdv \mathfrak a$, and hence $a^{-1}b \in a^{-1}bH = (a^{-1}H)bH = (a^{-1}H) \cdv (aH) \cdv \mathfrak a = \mathfrak a \subset H$. The argument for divisibility on the right side is similar.

If $\mathfrak a \in \mathcal I_v^* (H)$ and $a \in \mathfrak a \cap
\mathsf N (H)$, then $\mathfrak a \cdv \mathfrak a^{-1} = \mathfrak a^{-1} \cdv \mathfrak a = H$, $aH \subset \mathfrak a$, and hence $\mathfrak a
\cdv ( \mathfrak a^{-1} \cdv aH) = aH = ( aH \cdv \mathfrak a^{-1} ) \cdv \mathfrak a$. This shows that, if $\mathfrak
a \cap
      \mathsf N (H) \ne \emptyset$ for all $\mathfrak a \in \mathcal I_v^*
      (H)$, then $\mathcal P^{\mathsf n} (H) \subset \mathcal I_v^*
      (H)$ is cofinal. An analogous argument shows the converse.

\smallskip
4. If $H$ is normalizing, then $H = \mathsf N (H)$ is weakly normal.
Using 2., 3., and  Lemma \ref{4.6} we infer that
\[
\partial: H \overset{\pi}{\longrightarrow} H_{\red} \cong \mathcal P^{\mathsf n} (H) = \mathcal P
(H) \hookrightarrow \mathcal I_v^* (H)
\]
is a  cofinal divisor homomorphism, because it is a composition of
such homomorphisms.
\end{proof}

\smallskip
The following characterization of a divisor homomorphism will be
used without further mention.

\medskip
\begin{lemma} \label{4.8}
Let $\varphi \colon H \to D$ be a monoid homomorphism, and set $\phi
= \mathsf q (\varphi) \colon \mathsf q (H) \to \mathsf q (D)$. Then
the following statements are equivalent{\rm \,:}
\begin{enumerate}
\item[(a)] $\varphi$ is a divisor homomorphism.

\item[(b)] $\phi^{-1} (D) = H$.
\end{enumerate}
In particular, if $\varphi = (H \hookrightarrow D)$, then $H \subset
D$ is saturated if and only if $H = \mathsf q (H) \cap D$.
\end{lemma}

\begin{proof}
(a) \,$\Rightarrow$\, (b) \ Clearly, we have $H \subset \phi^{-1}
(D)$. If $x = a^{-1}b \in \phi^{-1} (D)$ with $a, b \in H$, then
$\phi (x) = \varphi (a)^{-1} \varphi (b) \in D$ and therefore
$\varphi (a) \tl \varphi (b)$. Hence $a \tl b$ and $x \in H$.

\smallskip
(b) \,$\Rightarrow$\, (a) \ Let $a, b \in H$ such that $\varphi (a)
\tl \varphi (b)$. Then $\phi (a^{-1}b) = \varphi (a)^{-1} \varphi
(b) \in D$, hence $a^{-1}b \in H$ and $a \tl b$. Similarly, $\varphi
(a) \tr \varphi (b)$ implies that $a \tr b$.

\smallskip
If $\varphi = (H \hookrightarrow D)$, then $\phi^{-1} (D) = \mathsf
q (H) \cap D$, and the assertion follows.
\end{proof}

\medskip
\begin{lemma} \label{4.9}
Let $D$ be a monoid and $H \subset D$ a saturated submonoid.
\begin{enumerate}

\item If $\mathfrak a \subset H$ is a left ideal of $H$, then $D
      \mathfrak a \subset D$ is a left ideal of $D$, and $D \mathfrak a
      \cap H = \mathfrak a$ $($similarly, if $\mathfrak a \subset H$ is a right ideal of
      $H$, then $\mathfrak a D \cap H = \mathfrak a$$)$.

\smallskip
\item Let $\mathfrak a \subset H$ be an ideal. If $\mathfrak a$ is a divisorial left ideal,
      then $\big(D \DPl (H \DPr \mathfrak a)\big)$ is a divisorial left
      ideal of $D$ with $\mathfrak a = \big(D \DPl (H \DPr \mathfrak a)\big) \cap H$. If $\mathfrak a$ is a divisorial right ideal,
      then $\big(D \DPr (H \DPl \mathfrak a)\big)$ is a
      divisorial right ideal of $D$ with $\mathfrak a = \big(D \DPr (H \DPl \mathfrak a)\big) \cap H$.

\smallskip
\item If $D$ satisfies the ascending chain condition on divisorial
      left ideals, then $H$ is $v$-noetherian.
\end{enumerate}
\end{lemma}

\smallskip
\noindent {\it Remark.} All quotients are formed in their respective
quotient groups. So $(H \DPr \mathfrak a) = \{ q \in \mathsf q (H)
\mid \mathfrak a q \subset H\}$, $\big(D \DPl (H \DPr \mathfrak
a)\big) = \{ q \in \mathsf q (D) \mid q (H \DPr \mathfrak a) \subset
D \}$, and so on.

\begin{proof}

\smallskip
1. Clearly, $D \mathfrak a \subset D$ is a left ideal of $D$, and we
have $\mathfrak a \subset D \mathfrak a \cap H$. If $x = uz \in H$
where $u \in D$ and $z \in \mathfrak a \subset H$, then $u \in
\mathsf q (H) \cap D = H$ and hence $x \in H \mathfrak a = \mathfrak
a$.

\smallskip
2. Let $\mathfrak a \subset H$ be a divisorial left ideal. Then $H
\subset (H \DPr \mathfrak a)$ and $D = HD \subset (H \DPr \mathfrak
a)D$ which implies that $\big(D \DPl (H \DPr \mathfrak a)\big) = (D
\DPl (H \DPr \mathfrak a)D ) \subset D$. By Lemma \ref{3.4}.6,
$\big(D \DPl (H \DPr \mathfrak a)\big)$ is a divisorial left ideal
of $D$.

If $a \in \mathfrak a$, then $a(H \DPr \mathfrak a) \subset H
\subset D$ and hence $a \in \big(D \DPl (H \DPr \mathfrak a)\big)$.
If $a \in \big(D \DPl (H \DPr \mathfrak a) \big) \cap H$, then $a (H
\DPr \mathfrak a) \subset D \cap \mathsf q (H) = H$ and hence $a \in
\big(H \DPl (H \DPr \mathfrak a)\big) = \mathfrak a$. Thus we have
$\mathfrak a = \big(D \DPl (H \DPr \mathfrak a)\big) \cap H$.

\smallskip
3. Let $(\mathfrak a_n)_{n \ge 0}$ be an ascending chain of
divisorial ideals of $H$, and set $\mathfrak A_n = \big(D \DPl (H
\DPr \mathfrak a_n)\big)$ for all $n \ge 0$. Then $(\mathfrak
A_n)_{n \ge 0}$ is an ascending chain of divisorial left ideals of
$D$. If it becomes stationary, then the initial chain $(\mathfrak
a_n)_{n \ge 0}$ becomes stationary because $\mathfrak a_n =
\mathfrak A_n \cap H$ for all $n \ge 0$.
\end{proof}

\medskip
\begin{lemma} \label{4.10}
Let  $\varphi \colon H \to D$ be a monoid homomorphism with $\varphi
(H) \subset \mathsf N (D)$, and set $\phi = \mathsf q (\varphi)
\colon \mathsf q (H) \to \mathsf q (D)$.
\begin{enumerate}
\item  If $H'$ is an overmonoid of $H$ with $aH'b \subset H$ for
       some $a, b \in H$, then $D' = D \phi (H')$ is an overmonoid of $D$
       with $\varphi (a)D' \varphi (b) \subset D$.

\smallskip
\item Suppose that $\varphi$ is a divisor homomorphism.
      \begin{enumerate}
      \smallskip
      \item If $D$ is completely integrally closed, then $H$ is
            completely integrally closed.

      \smallskip
      \item $H$ is normalizing.
      \end{enumerate}
\end{enumerate}
\end{lemma}

\begin{proof}
1. Since $\varphi (H) \subset \mathsf N (D)$, we have $D \phi (H') =
\phi (H')D$, and hence $D'$ is an overmonoid of $D$. Furthermore, we
get
\[
\varphi (a) D' \varphi (b) = \varphi (a) D \phi (H') \varphi (b) = D
\varphi (a) \phi (H') \varphi (b) = D \phi (aH'b) \subset D \,.
\]

\smallskip
2.(a) If $D$ is completely integrally closed and $H'$ is an
overmonoid of $H$ as in 1., then $H' \subset \phi^{-1} (D') =
\phi^{-1} (D) = H$. Thus $H$ is completely integrally closed by
Lemma \ref{3.10}.

\smallskip
2.(b) Let $a \in H$. We show that $aH \subset Ha$, and then by
symmetry we get $aH = Ha$. If $b \in aH$, then $\varphi (b) \in
\varphi (a)D = D \varphi (a)$, which implies that $\varphi (a) \tr
\varphi (b)$, $a \tr b$ and hence $b \in Ha$.
\end{proof}

\medskip
\begin{lemma} \label{4.11}
Let $\varphi \colon H \to D$ be a divisor homomorphism into a
normalizing monoid $D$, and set $\phi = \mathsf q (\varphi) \colon
\mathsf q (H) \to \mathsf q (D)$.
\begin{enumerate}
\item For every $X \subset H$ we have $X^{-1} = \phi^{-1} \big( \phi
(X)^{-1} \big)$.

\smallskip
\item For every $\mathfrak a \in \mathcal F_v (H)$ we have
$\mathfrak a = \phi^{-1} \big( \phi (\mathfrak a)_v \big)$.

\smallskip
\item If $D = \mathcal F (P)$, $\emptyset \ne \mathfrak a \in
\mathcal I_v (H)$ and $a = \gcd \big( \varphi (\mathfrak a) \big)$,
then $\mathfrak a = \varphi^{-1} ( aD)$.

\smallskip
\item Let $\varphi$ be a divisor theory.
      \begin{enumerate}
      \smallskip
      \item For every $a \in \mathsf q (D)$ there is a finite non-empty set $X \subset \mathsf q (H)$ such that $aD = \phi (X)_v$.

      \smallskip
      \item For every $\emptyset \ne X \subset
      H$, we have  $\gcd \big( \varphi (X) \big) = \gcd \big( \varphi (X_v) \big)$.
      \end{enumerate}
\end{enumerate}
\end{lemma}

\begin{proof}
We observe that  $H$ is normalizing by Lemma \ref{4.10}, and hence
$(H \DPl X) = (H \DPr X)$ for all $X \subset \mathsf q (H)$ by Lemma
\ref{4.5}.(b).  We will need the following fact for a commutative
monoid $M$ satisfying $\text{\rm GCD} (E) \ne \emptyset$ for all $E
\subset M$ (see \cite[Theorem 11.5]{HK98}): for any subset $X
\subset M$ we have
\[
X_v = dM \qquad \text{if and only if} \qquad \text{\rm GCD} (X) = dM^{\times} \,. \tag{$*$}
\]

\smallskip
1. If $x \in X^{-1}$, then $xX \subset H$, hence $\phi(x) \phi(X) =
\phi(xX) \subset D$, and $\phi(x) \in \phi(X)^{-1}$, which implies
$x \in \phi^{-1} \bigl(\phi(X)^{-1}\bigr)$.

Conversely, if $x \in \phi^{-1}\bigl(\phi(X)^{-1}\bigr)$, then
$\phi(xX) = \phi(x) \phi(X) \subset D$. Hence it follows that $xX
\subset \phi^{-1}(D) = H$ and $x \in X^{-1}$.

\smallskip
2. Let $\mathfrak a \in \mathcal F_v (H)$. Clearly, we have
$\mathfrak a \subset \phi^{-1} \big( \phi (\mathfrak a)_v \big)$.
Conversely, let $x \in \phi^{-1} \big( \phi (\mathfrak a)_v \big)$.
Then $\phi (x) \in \phi (\mathfrak a)_v = \big( \phi ( \mathfrak
a)^{-1} \big)^{-1}$, and hence by 1., we get
\[
\phi ( x \mathfrak a^{-1} ) = \phi \big( x \phi^{-1} ( \phi (
\mathfrak a)^{-1}) \big) \subset \phi (x) \phi ( \mathfrak a)^{-1}
\subset D \,.
\]
Since $H =  \phi^{-1} (D)$ by Lemma \ref{4.8}, it follows that $x
\mathfrak a^{-1} \subset H$ and thus $x \in (\mathfrak a^{-1})^{-1}
= \mathfrak a$.

\smallskip
3. If $a = \gcd \big( \varphi (\mathfrak a) \big)$, then $aD =
\varphi (\mathfrak a)_v$ by $(*)$, and 2.
implies that $\mathfrak a = \varphi^{-1} \big( \varphi (\mathfrak
a)_v \big) = \varphi^{-1} (aD)$.

\smallskip
4.  Suppose that   $D = \mathcal F (P)$.

4.(a) First we consider an element $a \in D$. Then $a = p_1 \cdot
\ldots \cdot p_l$ with $l \in \mathbb N_0$ and $p_1, \ldots, p_l \in
P$. For every $\nu \in [1, l]$ there exists a finite non-empty set
$X_{\nu} \subset H$ such that $p_{\nu} = \gcd \big( \varphi
(X_{\nu}) \big)$. Then the product set $X_1 \cdot \ldots \cdot X_l
\subset H$ is finite and $a = \gcd \big( \varphi (X_1 \cdot \ldots
\cdot X_l) \big)$ (where we use the convention that $X_1 \cdot
\ldots \cdot X_l = \{1\}$ if $l=0$). Now $(*)$ implies that $aD =
\varphi (X_1 \cdot \ldots \cdot X_l)_v$.

Let $a \in \mathsf q (D)$ be given. Then there is some $u \in H$
such that $\varphi (u)a \in D$. If $X \subset H$ is a finite
non-empty set with $\varphi (u) aD = \varphi (X)_v$, then $aD = \phi
(u^{-1}X)_v$.

\smallskip
4.(b) We start with the following assertion.

\begin{enumerate}
\item[{\bf A.}\,] For every $X \subset \mathsf q (H)$ we have $\phi (X)_v = \phi (X_v)_v$.
\end{enumerate}

Suppose that {\bf A} holds, let
$X \subset H$ and $a = \gcd \big( \varphi (X) \big)$. Applying {\bf A} and $(*)$ we infer that  $aD
= \varphi (X)_v = \varphi (X_v)_v$ and hence $a = \gcd \big( \varphi
(X_v) \big)$ by 3.

\smallskip
{\it Proof of} \,{\bf A}.\, Let $X \subset \mathsf q (H)$. Clearly,
we have $\phi (X)_v \subset \phi (X_v)_v$. To show the converse, we
assert that $\big(D \DP \phi (X) \big) \subset \big( D \DP \phi
(X_v) \big)$. This implies that
\[
\phi (X_v)_v = \big( D \DP \phi (X_v) \big)^{-1} \subset \big(D \DP
\phi (X) \big)^{-1} = \phi (X)_v \,.
\]
Let $a \in \big(D \DP \phi (X) \big) \subset \mathsf q (D)$. By
4.(a), there is a finite non-empty set $Y \subset \mathsf q (H)$
with $aD = \phi (Y)_v$. Then $\phi (XY) \subset \phi (X) aD \subset
D$ and hence $XY \subset H$. This implies that $X_v Y \subset (XY)_v
\subset H$, hence $\phi (X_v) \phi (Y) = \phi (X_v Y) \subset D$ and
therefore $\phi (X_v) \phi (Y)_v \subset \big( \phi (X_v) \phi (Y)
\big)_v \subset D$. Thus it follows that $\phi (X_v)a \subset \phi
(X_v) \phi (Y)_v \subset D$ and $a \in \big( D \DP \phi (X_v)
\big)$.
\end{proof}

\medskip
\begin{corollary} \label{4.12}
Let $\varphi \colon H \to D$ be a divisor homomorphism into a
normalizing monoid $D$.
\begin{enumerate}
\item If $D$ is $v$-noetherian, then $H$ is $v$-noetherian.

\smallskip
\item If $D$ is a Krull monoid, then $H$ is a normalizing Krull
monoid.
\end{enumerate}
\end{corollary}

\begin{proof}
1. If $(\mathfrak a_n)_{n \ge 0}$ is an ascending chain of
divisorial ideals of $H$, then $\big( \varphi (\mathfrak a_n)_v
\big)_{n \ge 0}$ is an ascending chain of divisorial ideals of $D$.
If this chain becomes stationary, then so does the initial chain in
$H$, because $\mathfrak a_n = \phi^{-1} \big( \phi (\mathfrak a_n)_v
\big)$ for all $n \ge 0$ by Lemma \ref{4.11}.2.

\smallskip
2. If $D$ is a normalizing Krull monoid, then $H$ is completely
integrally closed by Lemma \ref{4.10}.2, and hence the assertion
follows from 1.
\end{proof}

\medskip
\begin{theorem}[\bf A divisor theoretic characterization of normalizing Krull monoids]
\label{4.13}~

\noindent \ Let $H$ be a  monoid. Then the following statements are
equivalent{\rm \,:}
\begin{enumerate}
\item[(a)] The map $\partial \colon H \to \mathcal I_v^* (H)$, defined by
      $\partial (a) = aH$ for all $a \in H$, is a divisor theory.

\smallskip
\item[(b)] $H$ has a divisor theory.

\smallskip
\item[(c)] There exists a divisor homomorphism $\varphi \colon H \to
           \mathcal F (P)$ into a free abelian monoid.

\smallskip
\item[(d)] $H$ is a normalizing Krull monoid.

\end{enumerate}
\end{theorem}

\begin{proof}
(a) \,$\Rightarrow$\, (b) \,$\Rightarrow$\, (c) \ Obvious.

\smallskip
(c) \,$\Rightarrow$\, (d) \ Since $\mathcal F (P)$ is a normalizing
Krull monoid, this follows from Corollary \ref{4.12}.2.

\smallskip
(d) \,$\Rightarrow$\, (a) \ By Lemma \ref{4.7}.4, $\partial \colon H
\to \mathcal I_v^* (H)$ is a cofinal divisor
      homomorphism. Theorem \ref{3.14}
shows that  $\mathcal I_v^* (H)$ is a free abelian monoid with basis
$v$-$\spec (H) \setminus \{\emptyset\}$. Let $\mathfrak p$ be a
non-empty divisorial prime ideal. By Proposition \ref{3.13}.2, there
exists a finite set $E = \{a_1, \ldots, a_n \} \subset \mathfrak p$
such that $(HEH)_v = \mathfrak p$. Since $H$ is normalizing, we get
$HEH = a_1H \cup \ldots \cup a_nH$, where $a_1H, \ldots, a_nH$ are
divisorial ideals by Lemmas \ref{3.4} and \ref{4.5}. Now Proposition
\ref{3.12}.4 implies that
\[
\mathfrak p = (a_1H \cup \ldots \cup a_nH)_v = \gcd \big( \partial
(a_1), \ldots, \partial (a_n) \big) \,. \qedhere
\]
\end{proof}

\medskip
\begin{corollary} \label{4.14}
Let $H$ be a  monoid.
\begin{enumerate}
\item If $H$ is a Krull monoid, then $\mathsf N (H) \subset H$ is a normalizing Krull monoid, and
      there is a monomorphism $f \colon \mathcal I_v^* \big( \mathsf N (H) \big)
      \to \mathcal I_v^* (H)$ which maps $\mathcal P \big( \mathsf N (H) \big)$ onto $\mathcal P^{\mathsf n} (H)$.

\smallskip
\item $\mathsf N (H)$ is a normalizing Krull monoid if and only if $\mathsf N (H)_{\red}$ is a normalizing
      Krull monoid. If this holds, then both, $\mathsf N (H)_{\red} \cong \mathcal P^{\mathsf n} (H)$ and $\mathsf C (H)$, are
      commutative Krull monoids.
\end{enumerate}
\end{corollary}

\begin{proof}
We set $S = \mathsf N (H)$.

1. Suppose that $H$ is a Krull monoid. By Lemma \ref{4.3}.2, $S \subset H$ is a normalizing saturated
submonoid. Thus the inclusion map $S \hookrightarrow H$ satisfies
the assumption of Lemma \ref{4.10}.2, and hence $S$ is completely
integrally closed.

Let $f \colon \mathcal I_v^* ( S )
      \to \mathcal I_v^* (H)$ be defined by $f ( \mathfrak a) =
      \big( H \DPl (S \DPr \mathfrak a) \big)$ for all $\mathfrak a
      \in \mathcal I_v^* (S)$ (with the same notational conventions
      as in Lemma \ref{4.9}; in particular, $A = (S \DPr \mathfrak
      a) \subset \mathsf q (S)$).

We check that $f (\mathfrak a) \in \mathcal I_v^* (H)$. If $x \in
\mathsf q (H)$ with $xA \subset H$, then $xHA = xAH \subset H$, and
thus $(H \DPl A)$ is a right module of $H$. By Lemma \ref{4.9}.2,
$(H \DPl A)$ is a divisorial left ideal of $H$. Since $H$ is a Krull
monoid, it follows that $f ( \mathfrak a)$ is a divisorial ideal of
$H$, and hence $f (\mathfrak a) \in \mathcal I_v^* (H)$.

Since $f (\mathfrak a) \cap S = \mathfrak a$ by Lemma \ref{4.9}.2,
$f$ is injective and $S$ is $v$-noetherian because $H$ is
$v$-noetherian. If $a \in S$, then, by  Lemma \ref{3.4}.4, we infer
that
\[
f (Sa) = \big( H \DPl (S \DPr   Sa) \big) = ( H \DPl a^{-1}S ) = ( H
\DPl a^{-1}SH ) = Ha \,.
\]
This shows that $f$ maps $\mathcal P \big( S \big)$ onto $\mathcal
P^{\mathsf n} (H)$. Since $f_1 \colon \mathcal I_v^* (S) \to
\mathcal I_v^* (S)$, defined by $\mathfrak a \mapsto (S \DP
\mathfrak a)$, and $f_2 \colon \mathcal I_v^* (H) \to \mathcal I_v^*
(H)$, defined by $\mathfrak a \mapsto (H \DP \mathfrak a)$, are
homomorphisms, $f = f_2 \circ f_1$ (use Lemma \ref{3.6}) is a
homomorphism.

\smallskip
2. We freely use Theorem \ref{4.13}. If $S_{\red}$ is a normalizing Krull monoid, then there exists a divisor homomorphism $\varphi \colon S_{\red} \to \mathcal F (P)$. If $\pi \colon S \to S_{\red}$ denotes the canonical epimorphism, then $\varphi \circ \pi \colon S \to \mathcal F (P)$ is a divisor homomorphism by Lemma \ref{4.6} and thus $S$ is a normalizing Krull monoid. Suppose that $S$ is a normalizing Krull monoid. Again, by Theorem \ref{4.13}.(b) and by Lemma \ref{4.6}.3, it follows that $S_{\red}$ is a normalizing Krull monoid. Lemma \ref{4.7} shows that $S_{\red}$ and $\mathcal P^{\mathsf n} (H)$ are isomorphic, and that $\mathcal P^{\mathsf n} (H)$ is a submonoid of the commutative monoid $\mathcal I_v^* (H)$. Lemma \ref{4.3}.3 implies that $\mathsf C (H) \subset S$ is saturated, and thus $\mathsf C (H)$ is a Krull monoid  by Corollary \ref{4.12}.2.
\end{proof}

\medskip
Our next step is to introduce a concept of class groups, and then to
show a uniqueness result for divisor theories. Let $\varphi \colon H
\to D$ be a homomorphism of monoids. The group
\[
\mathcal C (\varphi) = \mathsf q (D) /\mathsf q \bigl(\varphi (H)
\bigr)
\]
is called the  {\it class group}  of $\varphi$. This coincides with
the notion in the commutative setting (see \cite[Section
2.4]{Ge-HK06a}), and we will point out that in case of a Krull
monoid $H$ and a divisor theory $\varphi \colon \mathsf N (H) \to D$
the class group $\mathcal C (\varphi)$ is isomorphic to the
normalizing class group of $H$ (see Equations (\ref{4.16}) and
(\ref{4.17}) at the end of this section).

For $a \in \mathsf q(D)$, we denote by
\[
[a]_\varphi = [a] = a \, \mathsf q\bigl( \varphi (H) \bigr) \in
\mathcal C (\varphi)
\]
the class containing \ $a$.  As usual, the class group \ $\mathcal
C(\varphi)$ \ will be written additively, that is,
\[
[ab] = [a] + [b] \quad \text{for all} \quad a,\,b \in \mathsf
q(D)\,,
\]
and then \,$[1] = 0$ is the zero element of $\mathcal C(\varphi)$.
If $\varphi \colon H \to D$ is a divisor homomorphism, then a
straightforward calculation shows that for an element $\alpha \in
D$, we have $[\alpha] = 0$ if and only if $\alpha \in \varphi (H)$.
If $D = \mathcal F (P)$ is free abelian, then $G_P = \{[p] \mid p
\in P \} \subset \mathcal C (\varphi)$ is the set of classes
containing prime divisors.

Consider the special case $H \subset D$, $\varphi = (H
\hookrightarrow D)$, and suppose that $\mathsf q (H) \subset \mathsf
q (D)$. Then $\mathcal C (\varphi) = \mathsf q (D)/ \mathsf q (H)$,
and we define
\[
D/H = \{ [a] = a \mathsf q (H) \mid a \in D \} \subset \mathcal C
(\varphi) \,.
\]
Then $D/H \subset \mathcal C (\varphi)$ is a submonoid with quotient
group $\mathcal C ( \varphi)$, and $D/H = \mathcal C ( \varphi)$ if
and only if $H \subset D$ is cofinal.

Suppose that $H$ is a normalizing Krull monoid, and let $\partial
\colon H \to \mathcal I_v^* (H)$ be as in Theorem \ref{4.13}. Then
$\mathcal P^{\mathsf n} (H) = \mathcal P (H) \subset \mathcal I_v^*
(H)$ is cofinal, and
\[
\mathcal C (\partial) = \mathcal I_v^* (H)/ \mathcal P (H) =
\mathcal F_v^{\times}(H)/ \mathsf q \big( \mathcal P (H) \big)
\]
is called the {\it $v$-class group} of $H$, and will be denoted by
$\mathcal C_v (H)$.

We continue with a uniqueness result for divisor theories. Its
consequences for class groups will be discussed afterwards. We
proceed as in the commutative case (\cite[Section 2.4]{Ge-HK06a}).
Recently, W.A. Schmid gave a  more explicit approach valid in case
of torsion class groups (\cite[Section 3]{Sc10a}).

\medskip
\begin{proposition}[\bf Uniqueness of Divisor Theories]
\label{4.15}~

Let $H$ be a monoid.
\begin{enumerate}
\item Let $\varphi \colon H \to F = \mathcal  F (P)$ be a divisor theory.
      Then the maps $\varphi^* \colon F \to \mathcal I_v^* (H)$ and
      $\overline \varphi \colon \mathcal C (\varphi)  \to \mathcal C_v (H)$, defined
      by
      \[
      \qquad \varphi^*(a) = \varphi^{-1} (aF)_v \quad \text{and} \quad
      \overline\varphi([a]_\varphi) = [\varphi^{-1} (aF)_v]  \ \text{ for
      all } \ a \in F \,,
      \]
      are isomorphisms.

\smallskip
\item If \ $\varphi_1 \colon H \to F_1$ and $\varphi_2 \colon H \to
F_2$ are divisor theories, then there is a unique isomorphism $\Phi
\colon F_1 \to F_2$ such that $\Phi \circ \varphi_1 = \varphi_2$. It
induces an isomorphism $\overline \Phi \colon \mathcal C(\varphi_1)
\to \mathcal C(\varphi_2)$, given by $\overline \Phi
([a]_{\varphi_1}) = [\Phi(a)]_{\varphi_2}$ for all $a \in F_1$.
\end{enumerate}
\end{proposition}

\begin{proof}
1. Note that $H$ is a normalizing Krull monoid by Theorem
\ref{4.13}. We start with the following assertion.

\begin{enumerate}
\item[{\bf A.}\,] $\{ \gcd \big( \varphi (X) \big) \mid \emptyset \ne X
\subset H \} = F$.
\end{enumerate}

{\it Proof of} \,{\bf A}.\, Since $\varphi \colon H \to \mathcal F
(P)$ is a divisor theory, it follows that $P \subset \big\{ \gcd
\big( \varphi (X) \big) \big| \emptyset \ne X \subset H \big\}$.
Since $\gcd \big( \varphi (X_1X_2) \big) = \gcd \big( \varphi (X_1)
\big) \gcd \big( \varphi (X_2) \big)$ for all non-empty subsets
$X_1, X_2 \subset H$, it follows that $\mathcal F (P) \subset \big\{
\gcd \big( \varphi (X) \big) \big| \emptyset \ne X \subset H \big\}
\subset \mathcal F (P)$.

Let $a \in F$. By {\bf A}, we have $a = \gcd \big( \varphi (X)
\big)$ for some non-empty subset $X \subset H$, and hence $\emptyset
\ne X \subset \varphi^{-1} (aF)$. This implies that $\varphi^{-1}
(aF)_v \in \mathcal I_v (H) \setminus \{\emptyset\} = \mathcal I_v^*
(H)$. By definition, we have $aF \cap \varphi (H) = \varphi \big(
\varphi^{-1} (aF) \big)$, and using Lemma \ref{4.11}.4 it follows
that
\[
a = \gcd \big(aF \cap \varphi (H) \big) = \gcd \big( \varphi (
\varphi^{-1}(aF)) \big) = \gcd \big( \varphi ( \varphi^{-1}(aF)_v)
\big) = \gcd \big( \varphi ( \varphi^* (a)) \big) \,,
\]
which shows that $\varphi^*$ is injective.

In order to show that $\varphi^*$ is surjective, let $\mathfrak a
\in \mathcal I_v^* (H)$ be given, and set $a = \gcd \big( \varphi
(\mathfrak a) \big)$. Then $\varphi^* (a) = \varphi^{-1} (aF)_v =
\mathfrak a$ by Lemma \ref{4.11}.3, and thus $\varphi^*$ is
surjective.

Next we  show that $\varphi^*$ is a homomorphism. Let $a,\, b \in
F$. Then Lemma \ref{3.6}.5 implies that
\[
\varphi^*(a) \cdv \varphi^* (b) = \bigl(
\varphi^{-1}(aF)_v\varphi^{-1} (bF)_v \bigr)_v = \bigl(
\varphi^{-1}(aF)\varphi^{-1} (bF) \bigr)_v \subset
\varphi^{-1}(abF)_v = \varphi^*(ab)\,.
\]
To prove the reverse inclusion, we set $c = \gcd \bigl ( \varphi
\bigl(\varphi^* (a)\cdv \varphi^*(b) \bigr)\bigr) \in F$, and note
that $\varphi^* (a) \cdv \varphi^* (b) \supset \varphi^{-1}(aF)
\varphi^{-1}(bF)$. This implies that
\[
c \t \gcd \Bigl( \varphi \bigl( \varphi^{-1}(aF) \varphi^{-1}(bF)
\bigr) \Bigr) = \gcd \bigl(aF \cap \varphi(H)\bigr) \gcd \bigl(bF
\cap \varphi(H) \bigr) = ab\,,
\]
hence $abF \subset cF$, and  thus $\varphi^* (ab) \subset
\varphi^{-1}(cF)_v = ( \varphi^*(a) \cdv \varphi^*(b))_v =
\varphi^*(a) \cdv \varphi^*(b)$, where the penultimate equation
follows from Lemma \ref{4.11}.3.

\smallskip
It remains to verify that $\overline \varphi$ is an isomorphism.
Note that for every $x \in H$, we have $\varphi ^* \circ \varphi(x)
= \varphi^{-1}\bigl( \varphi (x)F \bigr)_v = \mathsf q(\varphi)^{-1}
\bigl( \mathsf q (\varphi)(x) F \bigr)_v = xH$ by Lemma
\ref{4.11}.3. Obviously, $\varphi^*$ induces an epimorphism
$\varphi' \colon F \to \mathcal C_v (H)$, where $\varphi' (a) =
[\varphi^* (a)] \in \mathcal C_v (H)$. If $a,\, b \in F$ with
$[a]_\varphi = [b]_\varphi$, then there exist $x, \,y \in H$ such
that $\varphi(x)a = \varphi(y)b$. Since $[\varphi^*(a)] = [x
\varphi^*(a)] = \bigl[\varphi^* \bigl( \varphi(x)a \bigr)\bigr] =
\bigl[\varphi^* \bigl( \varphi(y)b \bigr)\bigr]= [y \varphi^*(b)]
=[\varphi^*(b)]$, it follows that $\varphi'$ induces an epimorphism
$\overline \varphi \colon \mathcal C (\varphi) \to \mathcal C_v
(H)$. To show that $\overline \varphi$ is injective, let $a, b \in
F$ with $[ \varphi^*(a) ] = [ \varphi^*(b)] \in \mathcal C_v(H)$.
Then there are $x,\, y \in H$ such that $x \varphi^*(a) = y
\varphi^*(b)$, hence $\varphi^*\bigl( \varphi(x)a \bigr) =
\varphi^*\bigl( \varphi(y)b \bigr)$,  thus $\varphi(x)a = \varphi(y)
b$, and therefore we get $[a]_\varphi = [b]_\varphi$.

\smallskip
2. For $i \in \{1,2\}$, let  $\varphi_i^* \colon F_i \to \mathcal
I_v^*(H)$ and $\overline \varphi_i \colon \mathcal C (\varphi_i) \to
\mathcal C_v(H)$ be the isomorphisms as defined in 1. Then $\Phi =
\varphi_2^{*-1} \circ \varphi_1^* \colon F_1 \to F_2$ and $\overline
\Phi = \overline \varphi_2^{-1} \circ \overline \varphi_1 \colon
\mathcal C(\varphi_1) \to \mathcal C(\varphi_2)$ are isomorphisms as
asserted.

Let $\psi \colon F_1 \to F_2$ be an arbitrary  isomorphism with the
property that $\psi \circ \varphi_1 = \varphi_2$. Then for every $a
\in F_1$ we have
\[
\psi(a) = \psi \bigl( \gcd \bigl( \varphi_1( \varphi_1^{-1}(aF_1))
\bigr) \bigr) = \gcd \bigl( \psi \circ \varphi_1 \bigl(
\varphi_1^{-1}(aF_1) \bigr) \bigr) = \gcd \bigl( \varphi_2 \bigl(
\varphi_1^{-1}(aF_1) \bigr) \bigr) \,,
\]
which shows that $\psi$ is uniquely determined.
\end{proof}

\medskip
Let $H$ be a Krull monoid and $\iota \colon \mathcal P^{\mathsf n}
(H) \hookrightarrow \mathcal I_v^* (H)$ be the inclusion map which
is a divisor homomorphism by Lemma \ref{4.7}.2. Then
\begin{equation}
\mathcal C^{\mathsf n} (H) = \mathcal C (\iota) \label{4.16}
\end{equation}
is called  the {\it normalizing class group} of $H$ (as studied by
Jespers and  Wauters, see \cite[page 332]{Je86a}). The
monomorphism $f \colon \mathcal I_v^* \big( \mathsf N (H) \big) \to
\mathcal I_v^* (H)$, discussed in Corollary \ref{4.14}, induces a
monomorphism
\[
\overline f \colon \mathcal C_v \big( \mathsf N (H) \big) = \mathcal
I_v^* \big( \mathsf N (H) \big)/ \mathcal P \big( \mathsf N (H)
\big) \to \mathcal C^{\mathsf n} (H) \,.
\]
In particular, if $H$ is normalizing and $\varphi \colon H \to D$ is
a divisor theory, then Proposition \ref{4.15} shows that
\begin{equation}
\mathcal C ( \varphi) \cong \mathcal C_v (H) = \mathcal C^{\mathsf
n} (H) \,, \label{4.17}
\end{equation}
and thus all concepts of class groups coincide.

\bigskip
\section{Examples of {K}rull monoids} \label{5}
\bigskip

In this section we provide a rough overview on the different places
where Krull monoids show up. We start with ring theory.

\smallskip
Let $R$ be a commutative integral domain. Then $R$ is a Krull domain
if and only if its multiplicative monoid of non-zero elements is a
Krull monoid. This was first proved independently by  Wauters
(\cite[Corollary 3.6]{Wa84a}) and  Krause (\cite{Kr89}). A
thorough treatment of this relationship and various generalizations
can be found in \cite[Chapters 22 and 23]{HK98} and \cite[Chapter
2]{Ge-HK06a}).  If $R$ is a Marot ring (this is a commutative ring
having not too many zero-divisors), then $R$ is a Krull ring if and
only if the monoid of regular elements is a Krull monoid
(\cite{HK93f}).

\smallskip
Next we consider the non-commutative setting.  A large number of concepts of
non-commutative Krull rings has been introduced (see \cite{Br73a,
Ma74a,Ma76a, Re77a, Re77b, Ma78a, Ch81a, Mi81a, Le83a, Wa-Je86a, Je-Wa86a, Je-Wa88a, Du91a}, and in particular  the survey article
\cite{Je86a}). Our definition of a Krull ring (given below) follows Jespers and  Okni{\'n}ski (\cite[page 56]{Je-Ok07a}). The following proposition summarizes the
relationship between the ideal theory of rings and the ideal theory
of the associated monoids of regular elements. This relationship was
first observed by  Wauters in \cite{Wa84a}. More detailed
references to the literature will be given after the proposition.
For clarity reasons, we carefully fix our setting for rings, and
then the proof of the proposition will be straightforward.

Let $R$ be a prime Goldie ring, and let $Q$ denote its classical
quotient ring (we follow the terminology of \cite{Mc-Ro01a} and
\cite{Go-Wa04a}; in particular, by a Goldie ring, we mean a  left
and right Goldie ring, and then the quotient ring is a left and
right quotient ring; an ideal is always a two-sided ideal). Then $Q$
is simple artinian, and every regular element of $Q$ is invertible.
Since $R$ is prime, every non-zero ideal $\mathfrak a \subset R$ is
essential, and hence it is generated as a left $R$-module (and also
as a right $R$-module) by its regular elements (see \cite[Corollary
3.3.7]{Mc-Ro01a}). By a fractional ideal $\mathfrak a$ of $R$ we
mean a left and right $R$-submodule of $Q$ for which there exist $a,
b \in Q^{\times}$ such that $a \mathfrak a \subset R$ and $\mathfrak
a b \subset R$. Clearly, every non-zero fractional ideal is generated by
regular elements. Let  $\mathfrak a$ be a fractional ideal. If
$\big(R \DPl (R \DPr \mathfrak a) \big) = \big(R \DPr (R \DPl
\mathfrak a) \big)$, then we set $\mathfrak a_v = \big(R \DPl (R
\DPr \mathfrak a) \big)$, and we say that $\mathfrak a$ is
divisorial if $\mathfrak a = \mathfrak a_v$. We denote by $\mathcal
F_v (R)$ the set of divisorial fractional ideals (fractional
$v$-ideals),  by $\mathcal I_v (R)$ the set of divisorial ideals of
$R$, and by $v$-$\spec (R)$ the set of divisorial prime ideals of
$R$. We say that $R$ is completely integrally closed if $(\mathfrak
a \DPl \mathfrak a) = (\mathfrak a \DPr \mathfrak a) = R$ for all
non-zero ideals $\mathfrak a$ of $R$. Suppose that $R$ is completely
integrally closed. Then left and right quotients coincide, and for
$\mathfrak a, \mathfrak b \in \mathcal F_v (R)$, we define
$v$-multiplication as $\mathfrak a \cdv \mathfrak b = (\mathfrak a
\mathfrak b)_v$. Equipped with $v$-multiplication, $\mathcal F_v
(R)$ is a semigroup, and $\mathcal I_v (R)$ is a subsemigroup. A prime Goldie ring is
said to be a {\it Krull ring} if it is completely integrally closed and satisfies
the ascending chain condition on divisorial ideals.

For a subset $I \subset Q$, we denote by $I^{\bullet} = I \cap
Q^{\times}$ the set of regular elements of $I$.  Then the set of all
regular elements  $H = R^{\bullet}$ of $R$ is a monoid, and $\mathsf
q (H) = Q^{\times}$ is a quotient group of $H$. Let $\mathfrak a,
\mathfrak b, \mathfrak c  $ be fractional ideals of $R$. Since
$\mathfrak c$ is generated (as a  left $R$-module and also as a
right $R$-module) by the regular elements, we have $\mathfrak c = \
_R\langle \mathfrak c^{\bullet} \rangle = \langle \mathfrak
c^{\bullet} \rangle_R$, and thus also
\[
(\mathfrak b \DPl \mathfrak a)^{\bullet} = (  \mathfrak b^{\bullet}
\DPl \mathfrak a^{\bullet} ) \quad \text{and} \quad  (\mathfrak b
\DPr \mathfrak a )^{\bullet}  =   (\mathfrak b^{\bullet} \DPr
\mathfrak a^{\bullet})  \,.
\]

\medskip
\begin{proposition} \label{5.1}
Let $R$ be a prime Goldie ring, and let $H$ be the monoid of regular
elements of $R$.
\begin{enumerate}
\item $R$ is completely integrally closed if and only if $H$ is
completely integrally closed.

\smallskip
\item The maps
      \[
      \iota^\bullet \colon
      \begin{cases}
      \mathcal F_v(R) &\to \  \mathcal F_v(H)\\
      \quad \mathfrak a &\mapsto \quad \mathfrak a^{\bullet}
      \end{cases}
      \qquad \text{and} \qquad \iota^\circ \colon
      \begin{cases}
      \mathcal F_v(H) &\to \  \mathcal F_v(R)\\
      \quad \mathfrak a &\mapsto \quad \langle \mathfrak a \rangle_R
      \end{cases}
      \]
      are inclusion preserving  isomorphisms which are inverse to each
      other. Furthermore,
      \begin{enumerate}
      \smallskip
      \item $\iota^{\bullet} \mid \mathcal I_v (R)  \colon \mathcal I_v (R)  \to \mathcal I_v (H)$ and
            $\iota^{\bullet} \mid v$-$\spec (R) \colon v$-$\spec (R) \to v$-$\spec
            (H)$ are bijections.

      \smallskip
      \item $R$ satisfies the ascending chain condition on  divisorial ideals of $R$ if
            and only if $H$ satisfies the ascending
            chain condition on divisorial ideals of $H$.
      \end{enumerate}

\smallskip
\item  $R$ is a Krull ring if and only if $H$ is a Krull monoid, and
      if this holds, then $\mathsf N (H)$ is a
      normalizing Krull monoid.
\end{enumerate}
\end{proposition}

\begin{proof}
1. Suppose that $H$ is completely integrally closed, and let
$\mathfrak a \subset R$ be a non-zero ideal. Then $\mathfrak
a^{\bullet} \subset H$ is an ideal,  $(\mathfrak a^{\bullet} \DPl
\mathfrak a^{\bullet})  = H$ by Lemma \ref{3.10} and hence
\[
(\mathfrak a \DPl \mathfrak a) = \ _R\langle (\mathfrak a \DPl
\mathfrak a)^{\bullet} \rangle = \ _R\langle ( \mathfrak a^{\bullet}
\DPl  \mathfrak a^{\bullet}) \rangle = \ _R\langle H \rangle = R \,.
\]
Similarly, we get $(\mathfrak a \DPr \mathfrak a) = R$.

Conversely, suppose that $R$ is completely integrally closed, and
let $\mathfrak a \subset H$ be a non-empty ideal. If $A \subset R$ denotes
the  ideal generated by $\mathfrak a$, then
\[
H \subset ( \mathfrak a \DPl \mathfrak a) \subset  (A \DPl
A)^{\bullet} = R^{\bullet} = H \,.
\]
Similarly, we get $(\mathfrak a \DPr \mathfrak a) = H$.

\smallskip
2. Clearly, $\iota^{\bullet}$ and $\iota^{\circ}$ are inclusion
preserving and map fractional ideals to fractional ideals. If
$\mathfrak a \in \mathcal F_v (R)$, then
\[
\big( H \DPl (H \DPr \mathfrak a^{\bullet}) \big) = ( R^{\bullet}
\DPl (R \DPr \mathfrak a)^{\bullet} \big) = \big( R \DPl (R \DPr
\mathfrak a) \big)^{\bullet} = \mathfrak a^{\bullet} = \big( R \DPr
(R \DPl \mathfrak a) \big)^{\bullet} = \big( H \DPr (H \DPl
\mathfrak a^{\bullet}) \big) \,,
\]
and hence $\mathfrak a^{\bullet}$ is a divisorial fractional ideal
of $H$. Similarly, we obtain that $\iota^{\circ} \big( \mathcal F_v
(H) \big) \subset \mathcal F_v (R)$. If $\mathfrak a \in \mathcal
F_v (R)$, then
\[
\iota^{\circ} \circ \iota^{\bullet} ( \mathfrak a) = \langle
\mathfrak a \cap Q^{\times} \rangle_R = \mathfrak a \,,
\]
and, if $\mathfrak a \in \mathcal F_v (H)$, then
\[
\iota^{\bullet} \circ \iota^{\circ} ( \mathfrak a) = \langle
\mathfrak a \rangle_R \cap Q^{\times} = \mathfrak a \,.
\]
Thus $\iota^{\bullet}$ and $\iota^{\circ}$ are inverse to each
other, and it remains to show that $\iota^{\bullet}$ is a
homomorphism.

Let $\mathfrak a, \mathfrak b, \mathfrak c \in \mathcal F_v (R)$. In
the next few calculations, we write---for clarity
reasons---$\mathfrak a \cdot_R \mathfrak b$ for the ring theoretical
product, $\mathfrak a \cdot_S \mathfrak b$ for the semigroup
theoretical product, $v_R$ for the $v$-operation on $R$ and $v_H$
for the $v$-operation on $H$. If $C \subset \mathfrak c^{\bullet}
\cap H$ is an ideal of $H$ such that $\langle C \rangle_R =
\mathfrak c$, then $(R \DPr \mathfrak c)^{\bullet} = (H \DPr C)$,
and hence
\[
\mathfrak c_{v_R} \cap Q^{\times} = \big( R \DPl (R \DPr \langle C
\rangle) \big)^{\bullet} = (R^{\bullet} \DPl (R \DPr \langle C
\rangle)^{\bullet} \big) = \big( H \DPl (H \DPr C) \big) = C_{v_H}
\,.
\]
Applying this relationship to $C = (\mathfrak a \cap
Q^{\times})\cdot_S(\mathfrak b \cap Q^{\times})$ we obtain that
\[
\begin{aligned}
\iota^{\bullet} ( \mathfrak a \cdot_{v_R} \mathfrak b) & = (
\mathfrak a \cdot_R \mathfrak b)_{v_R} \cap Q^{\times}  = \big(
\langle \mathfrak a \cdot_S \mathfrak b \rangle_R \big)_{v_R} \cap
Q^{\times}  = \big( \langle (\mathfrak a \cap Q^{\times}) \cdot_S
(\mathfrak b \cap
Q^{\times}) \rangle_R \big)_{v_R} \cap Q^{\times} \\
& = \big( (\mathfrak a \cap Q^{\times})\cdot_S(\mathfrak b \cap
Q^{\times}) \big)_{v_H}  = \iota^{\bullet}(\mathfrak a) \cdot_{v_H}
\iota^{\bullet} (\mathfrak b) \,.
\end{aligned}
\]

\smallskip
2.(a) It is clear that the restriction $\iota^{\bullet} \mid
\mathcal I_v (R)  \colon \mathcal I_v (R)  \to \mathcal I_v (H)$ is
bijective. We verify that $\iota^{\bullet} \mid v$-$\spec (R) \colon
v$-$\spec (R) \to v$-$\spec (H)$ is bijective. Indeed, if $\mathfrak
p \in v$-$\spec (R)$ and $\mathfrak a, \mathfrak b \in \mathcal I_s
(H)$ such that $\mathfrak a \mathfrak b \subset \mathfrak
p^{\bullet}$, then $\langle \mathfrak a \rangle_R \langle \mathfrak
b \rangle_R = \langle \mathfrak a \mathfrak b \rangle_R \subset
\mathfrak p$, whence $\langle \mathfrak a \rangle_R \subset
\mathfrak p$ or $\langle \mathfrak b \rangle_R \subset \mathfrak p$
and thus $\mathfrak a^{\bullet} \subset \mathfrak p^{\bullet}$ or
$\mathfrak b^{\bullet} \subset \mathfrak p^{\bullet}$. Therefore
$\mathfrak p^{\bullet}$ is a prime ideal by Lemma \ref{3.7}.(a), and
hence $\mathfrak p^{\bullet} \in  s$-$\spec (H) \cap \mathcal I_v
(H) = v$-$\spec (H)$. Conversely, suppose that $\mathfrak p \in
\mathcal I_v (R)$ such that $\mathfrak p^{\bullet} \in v$-$\spec
(H)$. In order to show that $\mathfrak p \subset R$ is a prime
ideal, let $\mathfrak a, \mathfrak b \subset R$ be ideals such that
$\mathfrak a \mathfrak b \subset \mathfrak p$. Then $\mathfrak
a^{\bullet} \mathfrak b^{\bullet} \subset (\mathfrak a \mathfrak
b)^{\bullet} \subset \mathfrak p^{\bullet}$, and thus $\mathfrak
a^{\bullet} \subset \mathfrak p^{\bullet}$ or $\mathfrak b^{\bullet}
\subset \mathfrak p^{\bullet}$, which implies that $\mathfrak a
\subset \mathfrak p$ or $\mathfrak b \subset \mathfrak p$.

\smallskip
2.(b) Since the restriction of $\iota^{\bullet}$ to $\mathcal I_v
(R)$ and the restriction of  $\iota^{\circ}$ to $\mathcal I_v (H)$
are both inclusion preserving and bijective, this follows
immediately.

\smallskip
3. The equivalence follows immediately from 1. and 2.(b). Moreover,
if $H$ is a Krull monoid, then $\mathsf N (H)$ is a normalizing
Krull monoid by Corollary \ref{4.14}.
\end{proof}

\smallskip
Suppose that  $R$ is a prime P.I.-ring. Then $R$ is a Krull ring
if and only if $R$ is a Chamarie-Krull ring (\cite[Proposition
3.5]{Wa84a}), and moreover the notions of $\Omega$-Krull rings,
central $\Omega$-Krull rings, Krull rings in the sense of
Marubayashi, in the sense of Chamarie and others coincide
(\cite[Theorem 2.4]{Je86a}). Classical orders in central simple
algebras over Dedekind domains are Asano prime rings (\cite[Theorem
5.3.16]{Mc-Ro01a}), and if $R$ is an Asano prime ring (in other
words, an Asano order), then $R$ is a Krull ring
(\cite[Proposition 5.2.6]{Mc-Ro01a}).
Moreover, if $R$ is a maximal order in a central simple algebra over a Dedekind domain with finite class group, then the central class group and hence the normalizing class group of $R$ are finite
(for more general results see \cite[Corollary 37.32]{R03}, \cite[Proposition 8.1]{Je86a}, \cite[Chapter E, Proposition 2.3]{Va-Ve84}). Krull rings, in which every element is normalizing, are discussed in \cite{Le-Oy86, Wa-Je86a}.
Further results and examples of
non-commutative Dedekind and Krull rings may be found in \cite{Ak-Bu11a, Wa-Je86a}.

\smallskip
If a monoid $H$ is normalizing, then every non-unit $a \in H$ is
contained in the divisorial ideal $aH \ne H$. But this does not hold
in general.  We provide the announced example of a Krull monoid $H$
having an element $a \in H \setminus H^{\times}$ which is not
contained in a divisorial ideal distinct from $H$ (we thank Daniel
Smertnig for his assistance).

\medskip
\begin{example} \label{5.2}
Let $R$ be a commutative principal ideal domain with quotient field $K$ and $n \in \mathbb N$. Then $M_n (R)$ is a classical order in the central simple algebra $M_n (K)$ and hence an Asano prime ring. By Proposition \ref{5.1}, $H = M_n (R)^{\bullet} = M_n (R) \cap \text{\rm GL}_n (K)$ is a Krull monoid with quotient group $\text{\rm GL}_n (K)$. Since every ideal of $M_n (R)$ is divisorial (\cite[Proposition 5.2.6]{Mc-Ro01a}),  we get
\[
\mathcal I_v (R) = \{ M_n (aR) \mid a \in R \} \,.
\]
Again by Proposition \ref{5.1}, this implies that
\[
\mathcal I_v (H) = \{ M_n (aR)^{\bullet} \mid a \in R \} \,,
\]
where
\[
M_n (aR)^{\bullet} = \{ C = (c_{i,j})_{1 \le i,j \le n} \mid c_{i,j} \in aR \ \text{for all} \ i, j \in [1,n] \ \text{and} \ \det (C) \ne 0 \} \,.
\]
Thus, if $C \in M_n (R)$ with $\text{\rm GCD} (\{c_{i,j} \mid i, j \in [1,n] \}) = R^{\times}$ and $\det (C) \ne 0$, then $(HCH)_v = H$.
\end{example}

\smallskip
We end this section with some more examples of Krull monoids.
Apart from their appearance as monoids of regular elements in Krull
rings, they occur in various other circumstances. We offer
a brief overview:
\begin{itemize}
\item Regular congruence monoids in Krull domains are Krull monoids
 (\cite[Proposition 2.11.6]{Ge-HK06a}).

\smallskip
\item Module Theory: Let $R$ be a
ring and $\mathcal C$  a class of right (or left)
$R$-modules---closed under finite direct sums, direct summands and
isomorphisms---such that $\mathcal C$ has a set $V ( \mathcal C)$ of
representatives (that is, every module $M \in \mathcal C$ is
isomorphic to a unique $[M] \in V( \mathcal C))$.  Then $V (
\mathcal C)$ becomes a commutative semigroup under the operation
$[M] + [N] = [M \oplus N]$, which carries detailed information about
the direct-sum behavior of modules in $\mathcal C$.  If every
$R$-module $M \in \mathcal C$ has a semilocal endomorphism ring,
then $\mathcal V (C)$ is a Krull monoid (see \cite{Fa02}, and
\cite{Fa06a} for a survey).

\smallskip
\item Diophantine monoids: A Diophantine monoid is a monoid which
consists of the set of solutions in nonnegative integers to a system
of linear Diophantine equations (see \cite[Proposition
4.3]{Ch-Kr-Oe02} and \cite[Theorem 2.7.14]{Ge-HK06a}).

\smallskip
\item Monoids of zero-sum sequences over abelian groups.
\end{itemize}

Since monoids of zero-sum sequences will be needed in the next
section, we discuss them in greater detail. Let $G$ be an additively
written abelian group and $G_0 \subset G$ a subset. The elements of
the free abelian monoid $\mathcal F(G_0)$ over $G_0$ are called \
{\it sequences over \ $G_0$}. Thus a sequence $S \in \mathcal
F(G_0)$ will be written in the form
\[
S = g_1 \cdot \ldots \cdot g_l = \prod_{g \in G_0} g^{\mathsf v_g
(S)}\,,
\]
and we use all notions (such as the length) as in general free
abelian monoids (see Section \ref{2}). Furthermore,  we denote by  \
$\sigma (S) = g_1+ \ldots + g_l$ \ the \ {\it sum} \ of $S$, and
\[
\mathcal B(G_0) = \{ S \in \mathcal F(G_0) \mid \sigma (S) =0\}
\]
is called the {\it monoid of zero-sum sequences} \ over \ $G_0$.
Clearly, $\mathcal B (G_0) \subset \mathcal F (G_0)$ is a saturated
submonoid, and hence it is a Krull monoid by Theorem \ref{4.13}.(b).
In Theorem \ref{6.5} we will outline the relationship between a
general Krull monoid and an associated monoid of zero-sum sequences.
An element $S = g_1 \cdot \ldots \cdot g_l$ is an atom in $\mathcal
B (G_0)$ if and only if it is a minimal zero-sum sequence (that is,
$\sigma (S) = 0$ but $\sum_{i \in I} g_i \ne 0$ for all $\emptyset
\ne I \subsetneq [1,l]$). The {\it Davenport constant}
\[
\mathsf  D (G_0) = \sup \bigl\{ |U| \, \bigm| \; U \in \mathcal A
\big( \mathcal B (G_0) \big) \bigr\} \in \N_0 \cup \{\infty\} \,,
\]
of $G_0$  is a central invariant in zero-sum theory (see
\cite{Ga-Ge06b}), and for its relevance in factorization theory we
refer to \cite{ Ge09a}. For a finite set $G_0$ we have $\mathsf D
(G_0) < \infty$ (see \cite[Theorem 3.4.2]{Ge-HK06a}).

\bigskip
\section{Arithmetic of Krull monoids} \label{6}
\bigskip

The theory of non-unique factorizations (in commutative monoids and
domains) has its origin in algebraic number theory, and in the last
two decades it emerged as an independent branch of algebra and
number theory (see \cite{An97,Ch-Gl00,Ch05a,Ge-HK06b, Ge-HK06a} for
some recent surveys and conference proceedings). Its main objective
is to describe the non-uniqueness of factorizations by arithmetical
invariants (such as sets of lengths,  defined below), and to study
the relationship between these arithmetical parameters and classical
algebraic parameters (such as class groups) of the rings under
investigation. Transfer homomorphisms play a crucial role in this
theory. They allow to shift problems from the original objects of
interest to auxiliary monoids, which are easier to handle; then one
has to settle the problems in the auxiliary monoids and shift the
answer back to the initial monoids or domains. This machinery is
best established---but not restricted to---in the case of
commutative Krull monoids, and it allows to employ methods from
additive and combinatorial number theory (\cite{Ge09a}).

In this section, we first show that the concept of a transfer
homomorphism carries over to the non-commutative setting in perfect
analogy. Then we give a criterion for  a Krull monoid to be a
bounded factorization monoid, and show that, if a Krull monoid
admits a divisor homomorphism with finite Davenport constant, then
all the arithmetical invariants under consideration are finite too
(Theorem \ref{6.5}). In order to do so we need all the ideal and
divisor theoretic tools developed in Sections \ref{3} and \ref{4}.

\smallskip
Let $H$ be a monoid. If $a \in H$ and $a = u_1 \cdot \ldots \cdot
u_k$, where $k \in \mathbb N$ and $u_1, \ldots, u_k \in \mathcal A
(H)$, then we say that $k$ is the {\it length} of the factorization.
For $a \in H \setminus H^{\times}$, we call
\[
\mathsf L_H (a) = \mathsf L (a) = \{ k \in \mathbb N \mid a \
\text{has a factorization of length} \ k \} \subset \N
\]
the {\it set of lengths} of $a$. For convenience, we set $\mathsf L
(a) = \{0\}$ for all $a \in H^{\times}$. By definition, $H$ is
atomic if and only if $\mathsf L (a) \ne \emptyset$ for all $a \in
H$. We say that $H$ is a \BF-{\it monoid} (or a bounded
factorization monoid) if $\mathsf L (a)$ is finite and non-empty for
all $a \in H$. We call
\[
\mathcal L (H) = \{ \mathsf L (a) \mid a \in H \}
\]
the {\it system of sets of lengths} of $H$. So if $H$ is a
\BF-monoid, then $\mathcal L (H)$ is a set of finite non-empty
subsets of the non-negative integers.

We recall some invariants describing the arithmetic of \BF-monoids.
Let $H$ be a \BF-monoid. If   $L = \{l_1, \ldots, l_t\} \subset
\mathbb N$, where $t \in \mathbb N$ and $l_1 < \ldots < l_t$, is a
finite non-empty subset of the positive integers,  then
\begin{itemize}
\smallskip
\item $\rho (L) = \frac{\max L}{\min L} \in \mathbb Q_{\ge 1}$ \ is
called the {\it elasticity} of $L$, and

\smallskip
\item $\Delta (L) = \{l_i - l_{i-1} \mid i \in [2, t] \}$ \ is called
the {\it set of distances} of $L$.
\end{itemize}
For convenience, we set $\rho ( \{0\}) = 1$ and $\Delta ( \{0\}) =
\emptyset$. We call
\begin{itemize}
\smallskip
\item $\rho (H) = \sup \{ \rho (L) \mid L \in \mathcal L (H) \} \in
      \mathbb R_{\ge 1} \cup \{\infty\}$ \ the {\it elasticity} of $H$, and

\smallskip
\item $\Delta (H) = \bigcup_{L \in \mathcal L (H)} \Delta (L) \
      \subset \mathbb N$ \ the {\it set of distances} of $H$.
\end{itemize}

Clearly, we have $\rho (H) = 1$ if and only if $\Delta (H) =
\emptyset$. Suppose that $\Delta (H) \ne \emptyset$, in other words
that there is some $L \in \mathcal L (H)$ such that $|L| \ge 2$.
Then there exists some $a \in H$ such that $a = u_1 \cdot \ldots
\cdot u_k = v_1 \cdot \ldots \cdot v_l$ where $1 < k < l$ and $u_1,
\ldots, u_k, v_1, \ldots, v_l \in \mathcal A (H)$. Then for every $n
\in \mathbb N$, we have
\[
a^n = (u_1 \cdot \ldots \cdot u_k)^{\nu}(v_1 \cdot \ldots \cdot
v_l)^{n - \nu} \quad \text{for all} \quad \nu \in [0, n]
\]
and hence $\{ln - \nu(l-k) \mid \nu \subset [0,n] \} \subset \mathsf
L (a^n)$. Therefore sets of lengths get arbitrarily large. We will
see that--under suitable algebraic finiteness conditions--sets of
lengths are well-structured. In order to describe their structure we
need the notion of almost arithmetical progressions.

Let $d \in \N$, \ $M \in \N_0$ \ and \ $\{0,d\} \subset \mathcal D
\subset [0,d]$. A subset $L \subset \Z$ is called an {\it almost
arithmetical multiprogression} \ ({\rm AAMP} \ for
      short) \ with \ {\it difference} \ $d$, \ {\it period} \ $\mathcal D$,
      \  and \ {\it bound} \ $M$, \ if
\[
L = y + (L' \cup L^* \cup L'') \, \subset \, y + \mathcal D + d \Z
\]
where \ $y \in \mathbb Z$ is a shift parameter,
\begin{itemize}
\item  $L^*$ is finite nonempty with $\min L^* = 0$ and $L^* =
       (\mathcal D + d \Z) \cap [0, \max L^*]$ and

\item  $L' \subset [-M, -1]$ \ and \ $L'' \subset \max L^* + [1,M]$
\end{itemize}

\smallskip
We say that \ {\it the Structure Theorem for Sets of Lengths} holds
for the monoid $H$ \ if $H$ is atomic and there exist some $M^* \in
\N_0$ and a finite nonempty set $\Delta^* \subset \N$ such that
every $L \in \mathcal L(H)$ is an {\rm AAMP} with some difference $d
\in \Delta^*$ and bound $M^*$ (in this case we say more precisely,
that the Structure Theorem holds with parameters $M^*$ and
$\Delta^*$).

\medskip
We start with a characterization of \BF-monoids, and for that we
need the notion of length functions. A function $\lambda \colon H
\to \mathbb N_0$ is called a {\it length function} if $\lambda (a) <
\lambda (b)$ for all $b \in (aH \cup Ha) \setminus (aH^{\times} \cup
H^{\times}a)$.

\medskip
\begin{lemma} \label{6.1}
Let $H$ be a monoid and $\mathfrak m = H \setminus H^{\times}$. Then
the following statements are equivalent{\rm \,:}
\begin{enumerate}
\item[(a)] $H$ is a \BF-monoid.

\smallskip
\item[(b)] $\bigcap_{n \ge 0} \mathfrak m^n = \emptyset$.

\smallskip
\item[(c)] There exists a length function $\lambda \colon H \to \mathbb
           N_0$.
\end{enumerate}
\end{lemma}

\begin{proof}
(a) \,$\Rightarrow$\, (b) \ Let $a \in \mathfrak m^k$ for some $k
\in \mathbb N$. Then there exist $a_1, \ldots, a_k \in \mathfrak m$
such that $a = a_1 \cdot \ldots \cdot a_k$ and hence $\max \mathsf L
(a) \ge k$. Since $\mathsf L (a)$ is finite, there exists some $l
\in \mathbb N$ such that $a \notin \mathfrak m^l \supset \bigcap_{n
\ge 0} \mathfrak m^n$.

\smallskip
(b) \,$\Rightarrow$\, (c) \ We define  a map $\lambda \colon H \to
\mathbb N_0$ by setting $\lambda (a) = \max \{ n \in \mathbb N_0
\mid a \in \mathfrak m^n \}$, and assert that $\lambda$ is a length
function. Let $a \in H$ and $b \in (aH \cup Ha) \setminus
(aH^{\times} \cup H^{\times}a)$, say $b \in aH$. Then $b = ac$ for
some $c \in \mathfrak m$. If $\lambda (a) = k$, then $a \in
\mathfrak m^k$, $b = ac \in \mathfrak m^{k+1}$, and thus $\lambda
(b) \ge k+1 > \lambda (a)$.

(c) \,$\Rightarrow$\, (a) \ Let $\lambda \colon H \to \mathbb N_0$
be a length function. Note that, if $b \in H^{\times}$ and $c \in H
\setminus H^{\times}$, then $c \in bH = H$ implies that $\lambda (c)
> \lambda (b) \ge 0$. We assert that every $a \in H \setminus
H^{\times}$ can be written as a product of atoms, and that $\sup
\mathsf L (a) \le \lambda (a)$. If $a \in \mathcal A (H)$, then
$\mathsf L (a) = \{1\}$, and the assertion holds. Suppose that $a
\in H$ is neither an atom nor a unit. Then $a$ has a product
decomposition of the form
\[
a = u_1 \cdot \ldots \cdot u_k \quad \text{ where} \quad  k \ge 2 \
\text{ and}  \ u_1, \ldots, u_k \in H \setminus H^{\times} \,.
\tag{$*$}
\]
For $i \in [0, k]$, we set $a_i = u_1 \cdot \ldots \cdot u_i$, and
then $a_{i+1} \in a_iH \setminus a_iH^{\times}$ for all $i \in [0,
k-1]$. This implies that $\lambda (a) = \lambda (a_k) > \lambda
(a_{k-1}) > \ldots > \lambda (a_1) > 0$ and thus $\lambda (a) \ge
k$. Therefore there exists a $k \in \N$ maximal such that $a = u_1
\cdot \ldots \cdot u_k$ where $u_1, \ldots, u_k \in H \setminus
H^{\times}$, and this implies that $u_1, \ldots, u_k \in \mathcal A
(H)$ and $k = \max \mathsf L (a) \le \lambda (a)$.
\end{proof}

\medskip
\begin{lemma} \label{6.2}
Let  $H$ be a monoid and $\Omega$  a set of prime ideals of $H$ such
that
\[
\bigcap_{n \in \N} \mathfrak p^n = \emptyset \quad \text{for all}
\quad \mathfrak p \in \Omega\,.
\]
If for every $a \in H \setminus H^{\times}$ the set $\Omega_a = \{
\mathfrak p \in \Omega \mid a \in \mathfrak p \}$ is finite and
non-empty, then $H$ is a \BF-monoid.
\end{lemma}

\begin{proof}
By Lemma \ref{6.1}, it suffices to show that $H$ has a length
function.  If $a \in H$ and $\Omega_a = \{\mathfrak p_1, \ldots,
\mathfrak p_k\}$, we define
\[
\lambda (a) = \sup \{ n_1 + \ldots + n_k \mid n_1, \ldots, n_k \in
\N_0, \ a \in \mathfrak p_1^{n_1} \cap \ldots \cap \mathfrak
p_k^{n_k} \} \,.
\]
By assumption, there exists some $n \in \N$ such that $a \notin
\mathfrak p_i^n$ for all $i \in [1,k]$, whence $\lambda (a) \le kn$.
We assert that $\lambda \colon H \to \N_0$ is a length function. Let
$a \in H$ and $b \in (aH \cup Ha) \setminus (aH^{\times} \cup
H^{\times}a)$, say $b = ac$ for some $c \in H \setminus H^{\times}$.
Since $\Omega_c \ne \emptyset$, there is a $\mathfrak q \in \Omega$
with $c \in \mathfrak q$. We assume that $\Omega_a = \{\mathfrak
p_1, \ldots, \mathfrak p_k\}$, \ $a \in \mathfrak p_1^{n_1} \cap
\ldots \cap \mathfrak p_k^{n_k}$ and $\lambda (a) = n_1+ \ldots
+n_k$. If $\mathfrak q \in \Omega_a$, say $\mathfrak q = \mathfrak
p_k$, then $b = ac \in (\mathfrak p_1^{n_1} \cap \mathfrak p_2^{n_2}
\cap \ldots \cap \mathfrak p_k^{n_k})\mathfrak p_k \subset \mathfrak
p_1^{n_1} \cap \mathfrak p_2^{n_2} \cap \ldots \cap \mathfrak
p_k^{n_k+1}$ and therefore $\lambda (b) \ge n_1 + \ldots +(n_k+1)
> \lambda (a)$. If $\mathfrak q \notin \Omega_a$,
then $b =ac \in ( \mathfrak p_1^{n_1} \cap \ldots \cap \mathfrak
p_k^{n_k})\mathfrak q \subset  \mathfrak p_1^{n_1} \cap \ldots \cap
\mathfrak p_k^{n_k} \cap \mathfrak q$ and thus again $\lambda (b)
\ge  n_1 + \ldots +n_k + 1 > \lambda (a)$.
\end{proof}

\medskip
\begin{definition} \label{6.3}
A monoid homomorphism \ $\theta \colon H \to B$ from a monoid $H$ onto a reduced monoid
$B$ is called a \ {\it transfer homomorphism} \ if it has the
following properties:

\smallskip

\begin{enumerate}
\item[]
\begin{enumerate}
\item[{\bf (T\,1)\,}] $B = \theta(H) $ \ and \ $\theta
^{-1} (1) = H^\times$.

\smallskip

\item[{\bf (T\,2)\,}] If $a \in H$, \ $b_1,\,b_2 \in B$ \ and \ $\theta
(a) = b_1b_2$, then there exist \ $a_1,\,a_2 \in H$ \ such that \ $a
= a_1a_2$, \ $\theta (a_1) = b_1$ \ and \ $\theta (a_2) =
b_2$.
\end{enumerate}\end{enumerate}
\end{definition}

\medskip
Transfer homomorphisms in a non-commutative setting were first used by  Baeth,  Ponomarenko et al. in \cite{B-P-A-A-H-K-M-R11}.

\medskip
\begin{proposition} \label{6.4}
Let $H$ and $B$ be  monoids, $\theta \colon H \to B$  a transfer
homomorphism and $a \in H$.

\smallskip

\begin{enumerate}
\item
If \ $k \in \N$, \ $b_1, \ldots, b_k \in B$ and $\theta(a) =
b_1 \cdot\ldots\cdot b_k$, then there exist $a_1, \ldots, a_k \in H$
such that $a = a_1 \cdot\ldots\cdot a_k$ and \ $\theta
(a_\nu) = b_\nu$ \ for all \ $\nu \in [1,k]$.

\smallskip

\item $a$ \ is an atom of \ $H$ if and only if \ $\theta (a)$ \
is an atom of \ $B$.

\smallskip

\item  $\mathsf L_H(a) = \mathsf L_B \big(\theta(a) \big)$.

\smallskip
\item $H$ is atomic $($a \BF-monoid resp.$)$ if and only if \ $B$ is atomic $($a \BF-monoid resp.$)$.

\smallskip
\item Suppose that $H$ is a \BF-monoid. Then $\rho (H) = \rho (B)$,
$\Delta (H) = \Delta (B)$, and the Structure Theorem for Sets of
Lengths holds for $H$ if and only if it holds for $B$ $($with the
same parameters$)$.
\end{enumerate}
\end{proposition}

\begin{proof}
1. This follows by induction on $k$.

\smallskip
2. Let $a \in H$ be an atom, and suppose that $\theta (a) = b_1b_2$
with $b_1, b_2 \in B$. By {\bf (T2)}, there exist $a_1, a_2 \in H$
with $a = a_1 a_2$ and $\theta (a_i) = b_i $ for $i \in [1, 2]$.
Since $a$ is an atom, we infer that $a_1 \in H^{\times}$ or $a_2 \in H^{\times}$, and thus
$b_1=1$ or $b_2  = 1$. Conversely, suppose that $\theta (a)$ is an
atom of $B$. If $a = a_1 a_2$, then $\theta (a) = \theta (a_1)
\theta (a_2)$. Thus $\theta (a_1) =1$ or $\theta (a_2) = 1$, and
therefore $a_1 \in H^{\times}$ or $a_2 \in H^{\times}$.

\smallskip
3. By {\bf (T1)}, it follows that $a \in H^{\times}$ if and only if $\theta (a) =
1$. Suppose that $a \notin H^{\times}$, and choose $k \in \N$. If $k \in \mathsf
L_H (a)$, then there exist $u_1, \ldots, u_k \in \mathcal A (H)$
such that $a = u_1 \cdot \ldots \cdot u_k$. Then $\theta (a) =
\theta (u_1) \cdot \ldots \cdot \theta (u_k)$. Since $\theta (u_1),
\ldots, \theta (u_k)  \in \mathcal A (B)$ by 2., it follows that $k
\in \mathsf L_B \big( \theta (a) \big)$. Conversely, suppose that $k
\in \mathsf L_B \big( \theta (a) \big)$. Then there are $b_1,
\ldots, b_k \in \mathcal A (B)$ such that $\theta (a) = b_1 \cdot
\ldots \cdot b_k$. Now 1. and 2.  imply that $k \in \mathsf L_H
(a)$.

\smallskip
4. A monoid $S$ is atomic (a \BF-monoid resp.) if and only if for
all $s \in S$, we have $\mathsf L (s) \ne \emptyset$ ($\mathsf L
(s)$ is finite and non-empty resp.). Thus the assertion follows from
3.

\smallskip
5. This follows immediately from 3. and 4.
\end{proof}

\medskip
\begin{theorem}[\bf Arithmetic of Krull monoids] \label{6.5}~

\noindent \ Let $H$ be a  Krull monoid.
\begin{enumerate}
\item If every $a \in H \setminus H^{\times}$ lies in a divisorial ideal distinct from $H$, then $H$ is a \BF-monoid.

\smallskip
\item Let  $\varphi \colon H \to D = \mathcal F (P)$ be a  divisor
      homomorphism, $G = \mathcal C (\varphi)$ its class group and $G_P
      \subset G$ the set of classes containing prime divisors.
      \begin{enumerate}
      \smallskip
      \item Let $\widetilde{\boldsymbol \beta} \colon \mathcal F (P) \to \mathcal F (G_P)$
            denote the unique homomorphism satisfying $\widetilde{\boldsymbol \beta} (p) = [p]$ for
            all $p \in P$. Then, for all $\alpha \in D$, we have
            $\widetilde{\boldsymbol \beta} (\alpha) \in \mathcal B (G_P)$
            if and only if $\alpha \in \varphi (H)$, and the map
            $\boldsymbol \beta = \widetilde{\boldsymbol \beta} \circ \varphi \colon H \to \mathcal B (G_P)$ is a transfer
            homomorphism.

      \smallskip
      \item If $\mathsf D (G_P) < \infty$, then $\rho (H) < \infty$,
            $\Delta (H)$ is finite, and there exists some $M^* \in \mathbb
            N_0$ such that the Structure Theorem for Sets of Lengths holds
            for $H$ with parameters $M^*$ and  $\Delta (H)$.
      \end{enumerate}
\end{enumerate}
\end{theorem}

\begin{proof}
1. We show that $\Omega = v$-$\spec (H) \setminus \{\emptyset\}$
satisfies the assumptions of Lemma \ref{6.2}. Then $H$ is a
\BF-monoid.

Let $a \in H \setminus H^{\times}$. By assumption, the set
$\Omega_a' = \{ \mathfrak a \in \mathcal I_v(H) \mid a \in \mathfrak
a \ \text{with} \ \mathfrak a \cap \{1\} = \emptyset\}$ is
non-empty, and  since $H$ is $v$-noetherian, $\Omega_a'$ has a
maximal element $\mathfrak p$ by Lemma \ref{3.13}, which is prime by
Lemma \ref{3.8}.1. Therefore the set $\Omega_a = \{ \mathfrak p \in
v\text{\rm -spec}(H) \mid a \in \mathfrak p\}$ is finite and
non-empty. Let $\mathfrak p \in v$-$\spec (H)$. If the intersection
of all powers of $\mathfrak p$ would be non-empty, it would be a
non-empty $v$-ideal and hence divisible by arbitrary powers of
$\mathfrak p$, a contradiction to the fact that $\mathcal I_v^* (H)$
is free abelian by Theorem \ref{3.14}.

\smallskip
2.(a) If $\alpha \in D$, then $\alpha = p_1 \cdot \ldots \cdot p_l$,
where $l \in \mathbb N_0$ and $p_1, \ldots, p_l \in P$,
$\widetilde{\boldsymbol \beta} (\alpha) = [p_1] \cdot \ldots \cdot
[p_l]$ and $\sigma \big( \widetilde{\boldsymbol \beta} (\alpha)
\big) = [p_1] + \ldots + [p_l] = [\alpha]$. Thus we have $[\alpha] =
0$ if and only if $\alpha \in \varphi (H)$. Therefore we obtain that
$\boldsymbol \beta = \widetilde{\boldsymbol \beta} \circ \varphi
\colon H \to \mathcal B (G_P)$ is a monoid epimorphism onto a
reduced monoid with $\boldsymbol \beta^{-1} (1) = H^{\times}$. To
verify {\bf (T2)}, let $a \in H$ with $\varphi (a) = p_1 \cdot
\ldots \cdot p_l \in D$, where $l \in \mathbb N_0$ and $p_1, \ldots,
p_l \in P$, and $\boldsymbol \beta (a) = [p_1] \cdot \ldots \cdot
[p_l] = b_1b_2$ with $b_1, b_2 \in \mathcal B (G_P)$. After
renumbering if necessary there is some $k \in [0, l]$ such that $b_1
= [p_1] \cdot \ldots \cdot [p_k]$ and $b_2 = [p_{k+1}] \cdot \ldots
\cdot [p_l]$. Setting $\alpha_1 = p_1 \cdot \ldots \cdot p_k$,
$\alpha_2 = p_{k+1} \cdot \ldots \cdot p_l$ we infer that $\alpha_1,
\alpha_2 \in \varphi (H)$, say $\alpha_i = \varphi (a_i)$ with $a_i
\in H$, and $\widetilde{\boldsymbol \beta} (\alpha_i) = b_i$ for $i
\in [1, 2]$. Then $\varphi (a) = \varphi (a_1)\varphi (a_2)$, and
hence by Lemma \ref{4.6}.2, we get $aH^{\times} = a_1 a_2
H^{\times}$. Thus there is an $\varepsilon \in H^{\times}$ such that
$a = (\varepsilon a_1)a_2$, $\boldsymbol \beta (\varepsilon a_1) =
\boldsymbol \beta (a_1) = b_1$ and $\boldsymbol \beta (a_2) = b_2$.

\smallskip
2.(b) Suppose that $\mathsf D (G_P) < \infty$. By Proposition
\ref{6.4}.5, it suffices to prove all assertions for the monoid
$\mathcal B (G_P)$. Thus the finiteness of the elasticity and of the
set of distances follows from \cite[Theorem 3.4.11]{Ge-HK06a}, and
the validity of the Structure Theorem follows from \cite[Theorem
5.1]{Ge-Ka10a} or from \cite[Theorem 4.4]{Ge-Gr09b}.
\end{proof}

\bigskip
\noindent {\bf Acknowledgement.} I would like to thank Franz
Halter-Koch and Daniel Smertnig, who read a previous version of this
manuscript and made many valuable suggestions. Furthermore, I wish to thank Jan Okni{\'n}ski, who made me aware of several results on semigroup algebras, and also the referee for reading the paper so carefully and for providing a list of useful comments.

\bigskip
\providecommand{\bysame}{\leavevmode\hbox to3em{\hrulefill}\thinspace}
\providecommand{\MR}{\relax\ifhmode\unskip\space\fi MR }
\providecommand{\MRhref}[2]{%
  \href{http://www.ams.org/mathscinet-getitem?mr=#1}{#2}
}
\providecommand{\href}[2]{#2}

\end{document}